\documentclass{article}
\usepackage{latexsym}
\usepackage{leftidx}
\usepackage{amsmath,amsthm, pdfpages}
\usepackage{enumerate}
\usepackage{amssymb}
\usepackage{upgreek}
\usepackage[margin=1.3in,dvips]{geometry}

\newcommand{\vol}{\textnormal{vol}}

\def\f12{\frac 1 2}

\def\a{\alpha}
\def\b{\beta}
\def\ga{\gamma}

\def\ep{\epsilon}

\def\si{\sigma}
\def\Si{\Sigma}
\def\om{\omega}

\def\Lb{\underline{L}}

\def\pa{\partial}
\def\les{\lesssim}

\newcommand{\D}{\mbox{$D \mkern-13mu /$\,}}

\newtheorem{thm}{Theorem}

\newtheorem{lem}{Lemma}

\newtheorem{remark}{Remark}
\begin{document}

\title{On global behavior of solutions of the Maxwell-Klein-Gordon equations}
\date{}
\author{Shiwu Yang}
\maketitle

\begin{abstract}
 It is known that the Maxwell-Klein-Gordon equations in $\mathbb{R}^{3+1}$ admit global solutions with finite energy data. In this paper, we present a new approach to study the asymptotic behavior of these global solutions. We show the quantitative energy flux decay of the solutions with data merely bounded in some weighted energy space. We also establish an integrated local energy decay and a hierarchy of $r$-weighted energy decay. The results in particular hold in the presence of large total charge. This is the first result to give a complete and precise description of the global behavior of large nonlinear charged scalar fields.
\end{abstract}

\section{Introduction}
In this paper, we study the asymptotic behavior of solutions to the Maxwell-Klein-Gordon equations on $\mathbb{R}^{3+1}$ with large Cauchy data. To define the equations, let $A=A_\mu dx^\mu$ be a $1$-form. The covariant derivative associated to this 1-form is
\begin{equation*}
D_\mu =\pa_\mu+\sqrt{-1}A_\mu,
\end{equation*}
which can be viewed as a $U(1)$ connection on the complex line bundle over $\mathbb{R}^{3+1}$ with the standard flat metric $m_{\mu\nu}$. Then the curvature $2$-form $F$ is given by
\begin{equation*}
F_{\mu\nu}=-\sqrt{-1}[D_{\mu}, D_{\nu}]=\pa_\mu A_\nu-\pa_\nu A_\mu=(dA)_{\mu\nu}.
\end{equation*}
This is a closed $2$-form, that is, $F$ satisfies the Bianchi identity
\begin{equation}
\label{bianchi}
 \pa_\ga F_{\mu\nu}+\pa_\mu F_{\nu\ga}+\pa_\nu F_{\ga\mu}=0.
\end{equation}
The Maxwell-Klein-Gordon equations (MKG) is a system for the connection field $A$ and the complex scalar field $\phi$:
\begin{equation}
 \label{EQMKG}\tag{MKG}
\begin{cases}
\pa^\nu F_{\mu\nu}=\Im(\phi \cdot\overline{D_\mu\phi})=J_\mu;\\
D^\mu D_\mu\phi=\Box_A\phi=0.
\end{cases}
\end{equation}
These are Euler-Lagrange equations of the functional
\[
L[A, \phi]=\iint_{\mathbb{R}^{3+1}}\frac{1}{4}F_{\mu\nu}F^{\mu\nu}+\frac{1}{2}D_{\mu}\phi\overline{D^{\mu}\phi}dxdt.
\]
A basic feature of this system is that it is invariant under the following gauge transformation:
\[
 \phi\mapsto e^{i\chi}\phi; \quad A\mapsto A-d\chi.
\]
More precisely, if $(A, \phi)$ solves \eqref{EQMKG}, then $(A-d\chi, e^{i\chi}\phi)$ is also a solution for any potential function $\chi$. Note that $U(1)$ is abelian. The Maxwell field $F$ is invariant under the above gauge transformation and \eqref{EQMKG} is said to be an \textsl{abelian gauge theory}. For the more general theory when $U(1)$ is replaced by a general compact Lie group, the corresponding equations are referred to as \textsl{Yang-Mills-Higgs equations}.

\bigskip

In this paper, we consider the Cauchy problem to \eqref{EQMKG}. The initial data set $(E, H, \phi_0, \phi_1)$ consists of the initial electric field $E$, the magnetic field $H$, together with initial data $(\phi_0, \phi_1)$ for the scalar field. In terms of the solution $(F, \phi)$, on the initial hypersurface, these are:
\begin{equation*}
F_{0i}=E_i,\quad \leftidx{^*}F_{0i}=H_i,\quad \phi(0, x)=\phi_0,\quad D_t\phi(0, x)=\phi_1,
\end{equation*}
where $\leftidx{^*}F$ is the Hodge dual of the 2-form $F$. In local coordinates $(t, x)$,
\[
(H_1, H_2, H_3)=(F_{23}, F_{31}, F_{12}).
\]
 The data set is said to be \textsl{admissible} if it satisfies the compatibility condition
\begin{equation*}
div(E)=\Im(\phi_0\cdot \overline{\phi_1}),\quad div (H)=0,
\end{equation*}
where the divergence is taken on the initial hypersurface $\mathbb{R}^3$. For solutions of \eqref{EQMKG}, the energy
\[
 E[F, \phi](t):=\int_{\mathbb{R}^3}|E|^2+|H|^2+|D\phi|^2dx
\]
is conserved. Another important conserved quantity is the total charge
\begin{equation}
\label{defcharge}
q_0=\frac{1}{4\pi}\int_{\mathbb{R}^3}\Im(\phi\cdot \overline{D_t\phi})dx=\frac{1}{4\pi}\int_{\mathbb{R}^3}div (E)dx,
\end{equation}
which can be defined at any fixed time $t$. The existence of nonzero charge plays a crucial role in the asymptotic behavior of solutions of \eqref{EQMKG}. It makes the analysis more complicated and subtle. This is obvious from the above definition as the electric field $E_i=F_{0i}$ has a tail $q_0r^{-3}x_i$ at any fixed time $t$.

\bigskip

The Cauchy problem to \eqref{EQMKG} has been studied extensively. One of the most remarkable results is due to Eardley-Moncrief in \cite{Moncrief1}, \cite{Moncrief2}, in which it was shown that there is always a global solution to the general Yang-Mills-Higgs equations for sufficiently smooth initial data.
 This result has later been greatly improved for MKG equations by Klainerman-Machedon in their celebrated work \cite{MKGkl}. We state their result here as it guarantees a global solution to study in this paper.
\begin{thm}
\label{thm:KMKG}
 There is a unique (up to gauge transformations) global solution of \eqref{EQMKG} if the initial data are bounded in the energy space, that is, $E[F, \phi](0)$ is finite.
\end{thm}
Same statement holds for the non-abelian case of Yang-Mills equations, see e.g. \cite{YMkl}, \cite{sungjinYM},
\cite{MKGtesfahun}. Since then there has been an extensive literature on generalizations and extensions of the above classical result, aiming at improving the regularity of the initial data in order to construct a global solution. For more details, we refer to \cite{MKGtao}, \cite{KriegerMKG4}, \cite{MKGtataru}, \cite{MKGmachedon}, \cite{OhMKG4}, \cite{MKGigor} and reference therein. A common feature of all these works is to
 construct a local solution with rough data. Then the global well-posedness follows by establishing a priori bound for some appropriate norms of the solution. In the work of Eardley-Moncrief, based on the conservation law of energy, they showed that the $L^\infty$ norm of the solution never blows up even though it may grow in time $t$. As a consequence the solution can be extended to all time; however the decay property of the solution is unknown. For Theorem \ref{thm:KMKG}, the difficulty is to construct a solution with finite energy data on a short time interval so that the length of the time interval depends only on the energy norm. Then the solution exists globally in time since the energy is conserved. It is not clear how (if it is possible) this construction of global solution is able to give any asymptotics of the solution. In view of this, although the solution of \eqref{EQMKG} exists
globally with rough initial data, very little is known about the decay properties.

\bigskip

On the other hand, asymptotic behavior and decay estimates are only understood for linear fields (see e.g.\cite{asymLkl}) and nonlinear fields with sufficiently small initial data (see e.g.\cite{fieldschrist}, \cite{Shu}). These mentioned results rely on the conformal symmetry of the system, either by conformally compactifying the Minkowski space or by using the conformal killing vector field $(t^2+r^2)\pa_t+2tr\pa_r$ as multiplier. Nevertheless the use of the conformal symmetry requires strong decay of the initial data and thus in general does not allow the presence of nonzero charge except when the initial data are essentially compactly supported. For the case with nonzero charge but still with small initial data, the first related work regarding the asymptotic properties was due to W. Shu in \cite{shu2}. However, that work only considered the case when the solution is trivial outside
a fixed forward light cone. Details for general case were not carried out. A complete proof towards this program was later contributed by Lindblad-Sterbenz in \cite{LindbladMKG}, also see a more recent work \cite{MKGLydia}.

 \bigskip

 The presence of nonzero charge has a long range effect on the asymptotic behavior of the solutions, at least in a neighbourhood of the spatial infinity. This can be seen from the conservation law
of the total charge as the electric field $E$ decays at most $ r^{-2}$ as $r\rightarrow\infty$ at any fixed time. This weak decay rate makes the analysis more complicated even for small initial data. To deal with this difficulty, Lindblad-Sterbenz constructed a global chargeless field and made use of the fractional Morawetz estimates obtained by using $u^p\pa_u+v^p\pa_v$ as multipliers. The latter work \cite{MKGLydia} relied on the observation that the angular derivative of the Maxwell field has zero charge. The Maxwell field then can be estimated by using Poincar\'e inequality. However, the smallness
of the initial data, in particular the smallness of the charges, still allows one to use perturbation method to obtain the decay estimates for the solutions.

\bigskip

From the above discussion, we see that on one hand the solution of \eqref{EQMKG} exists globally with data merely in energy space and on the other hand, the asymptotic behavior is only clear for small data. It is then natural to ask whether the global nonlinear charged-scalar-fields (solutions of MKG) exhibit any form of decay properties.
The aim of the present paper is to affirmatively answer this question. We show strong quantitative decay estimates for solutions of \eqref{EQMKG} with data merely bounded in some weighted energy space.

\bigskip

To be more precise, assume that the initial data $(E, H, \phi_0, \phi_1)$ belong to the weighted energy space with weights $(1+r)^{1+\ga_0}$ for some positive constant $\ga_0$. In particular, the energy is finite. According to Theorem \ref{thm:KMKG}, there is a global solution $(F, \phi)$ of \eqref{EQMKG}. Let $\tilde{F}$ be the chargeless field\footnote{ Compared to the construction in \cite{LindbladMKG}, we avoid the use of a smooth cut off function. The reason is that we will carry out estimates respectively in the exterior region $\{t+R<r\}$ and interior region $\{t+R>r\}$, where the field $\tilde{F}$ is smooth.}:
\[
\tilde{F}=F-q_0 r^{-2}\chi_{\{r\geq t+R\}}dt\wedge dr,\quad R>1,
\]
where $q_0$ is the total charge and $\chi_{K}$ is the characteristic function on the set $K$. We can show that the energy flux through the outgoing light cone $\Si_{\tau}$ decays pointwise in terms of $\tau$:
\[
 E[\tilde{F}, \phi](\Si_{\tau})\leq C(1+|\tau|)^{-1-\ga}
 \]
for all $0<\ga<\ga_0$ and some constant $C$ depending on $\ga_0$, $\ga$ and the initial weighted energy. Moreover, we can establish an integrated local energy decay and a class of $r$-weighted energy decay. These decay estimates are sufficiently strong to capture the asymptotic properties of solutions of \eqref{EQMKG}. In fact, the above energy flux decay immediately leads to the pointwise decay for the spherical average of the scalar field. In our future works, these decay estimates are the first crucial step toward the stronger pointwise decay for large nonlinear charged-scalar-fields.

\bigskip

Furthermore, since the data are assumed to be merely bounded in some weighted energy space, the total charge $q_0$ is generic and can be arbitrarily large. Note that $\tilde{F}=F$ inside the forward light cone $\{t+R\geq r\}$. Our result in particular shows that the charge can only affect the asymptotic behavior of the solution outside this fixed light cone. This phenomenon was conjectured
by W. Shu in \cite{Shu} and has been confirmed for sufficiently smooth and small initial data in \cite{MKGLydia}, \cite{LindbladMKG}. We thus give an affirmative answer to the conjecture for all data in the weighted energy space described above. In other words, the quantitative decay estimates together with the precise description of the effect of the charge completely characterize the asymptotic behavior of solutions of \eqref{EQMKG}.

\bigskip

To make the statement of the main result precise, we define some necessary notations.
We use the standard polar local coordinate system $(t, r,
\om)$ of Minkowski space as well as the null coordinates $u=\frac{t-r}{2}$, $v=\frac{t+r}{2}$.
Without loss of generality we only prove estimates in the future, i.e., $t\geq 0$. In our argument, we estimate the decay of the solution with respect
to the foliation $\Si_{\tau}$
\begin{figure}[h]
\centering
\includegraphics[trim=0mm 120mm 0mm 0mm, scale=0.25]{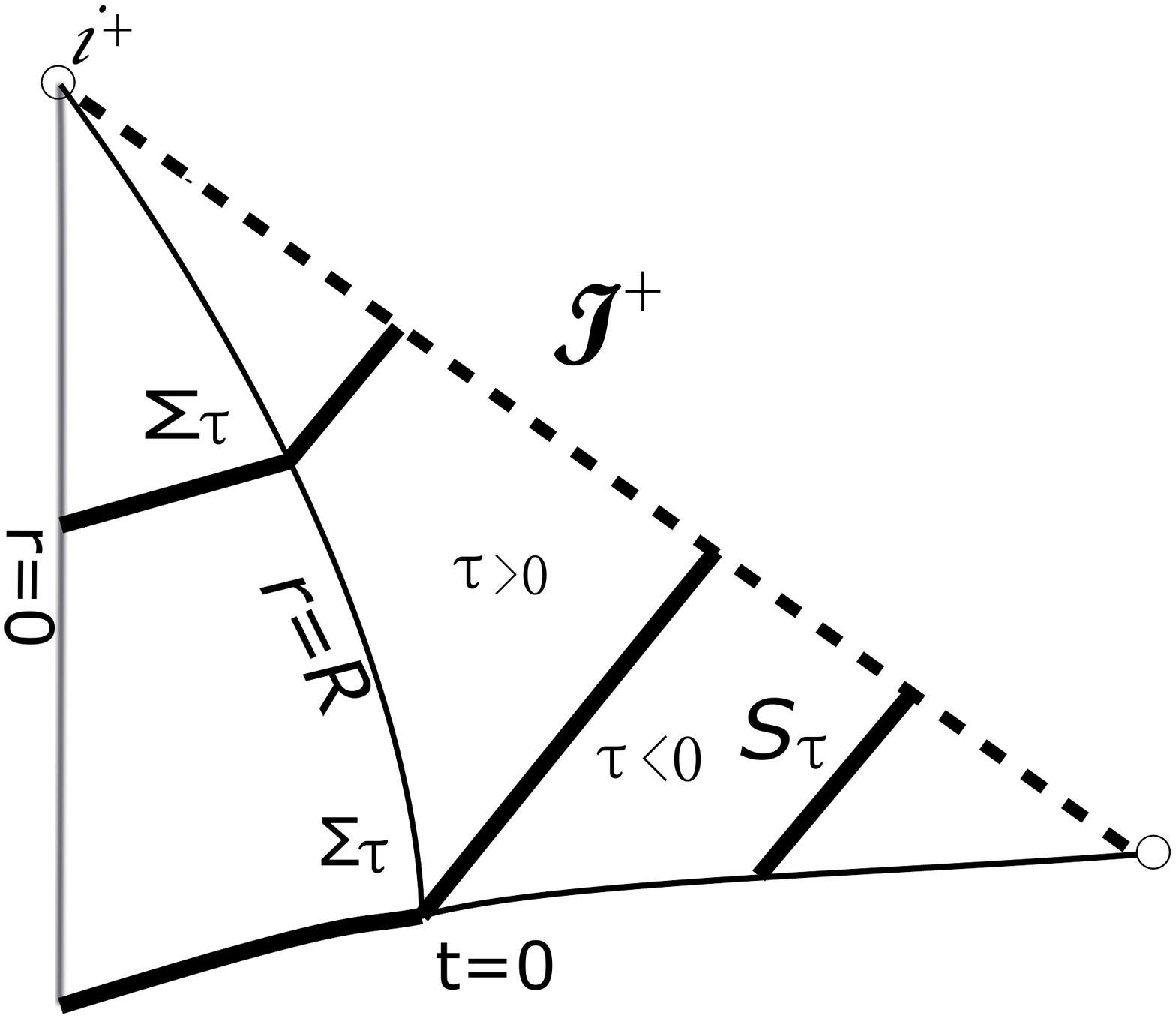}
\end{figure}

\noindent depicted in the above Penrose diagram.
Here $\tau<0$ corresponds to the foliation in the exterior region $\{t+R\leq r\}$ and $\tau\geq 0$ foliates the interior region.
In coordinates the leaves $\Si_{\tau}$ can be defined rigorously as follows:
 \begin{align*}
&S_\tau:=\{u=u_\tau=\f12(\tau-R), v\geq v_\tau=\f12(|\tau|+R)\},\quad \tau\in \mathbb{R};\\
&\Si_\tau:=\{t=\tau, r\leq R\}\cap\{t\geq 0\}\cup S_\tau,\quad \tau\in\mathbb{R}
\end{align*}
for some constant $R>1$.

 Next we introduce a null frame $\{L, \Lb, e_1, e_2\}$, where
\[
L=\pa_v=\pa_t+\pa_r,\quad \Lb=\pa_u=\pa_t-\pa_r
\]
and $\{e_1, e_2\}$ is an orthonormal basis of the sphere with
constant radius $r$. We use $\D$ to denote the covariant derivative associated to the connection field $A$ on the sphere with radius $r$. For any 2-form $F$, denote the null decomposition under the above null frame by
\begin{equation}
\a_i=F_{Le_i},\quad\underline{\a}_i=F_{\Lb e_i},\quad \rho=\f12 F_{\Lb L}, \quad \si=F_{e_1 e_2},\quad i\in{1, 2}.
\end{equation}
Let $E[\phi, F](\Si)$ be the energy flux of the 2-form $F$ and the complex scalar field $\phi$ through the hypersurface $\Si$ in Minkowski space. For the outgoing null hypersurface $S_{\tau}$, one can compute that
\begin{align*}
 E[\phi, F](S_\tau)&=\int_{S_\tau}(|D_L\phi|^2+|\D\phi|^2+\rho^2+\si^2+|\a|^2) r^2dvd\om.
\end{align*}
 For the admissible initial data set $(E, H, \phi_0, \phi_1)$ for \eqref{EQMKG}, let $q_0$ be the total charge. Define the chargeless electric field $\tilde{E}$ on the initial hypersurface:
\[
\tilde{E}=E-q_0 r^{-2}\chi_{\{R\leq r\}} \om.
\]
The constant $R>1$ was used to define the foliation $\Si_{\tau}$. Our assumption on the initial data is that for some positive constant $0<\ga_0\leq 1$ the following weighted energy
\begin{equation}
\label{IDMKG}
E_0= \int_{\mathbb{R}^3}(1+r)^{1+\ga_0}(|D\phi_0|^2+|\phi_1|^2+|\tilde{E}|^2+|H|^2)dx
\end{equation}
is finite. We now can state our main theorem:
\begin{thm}
 \label{energydecayMKG}
Consider the Cauchy problem to \eqref{EQMKG}. Assume the initial data set $(E, H, \phi_0, \phi_1)$ with charge $q_0$ is admissible
and the weighted energy $E_0$ is finite for some positive constant $\ga_0\leq 1$. Then for all $\ep>0$, $0\leq \ga<\ga_0$, $R>1$ the global solution $(F, \phi)$ obeys the following decay estimates:
\begin{itemize}
\item[1.] The energy flux decay and the integrated local energy decay:
\[
 E[\phi, \tilde{F}](\Si_{\tau})+\int_{\frac{s}{\tau}\geq 1}\int_{\Si_s }\frac{|\tilde{F}|^2+|D\phi|^2}{(1+r)^{1+\ep}}dxds\leq CE_0(1+|\tau|)^{-1-\ga},\quad \forall\tau\in\mathbb{R};
\]
\item[2.] A hierarchy of $r$-weighted energy estimates for all $0\leq p\leq 1+\ga$:
\begin{align*}
 &\int_{\frac{s}{\tau}\geq 1}\int_{S_s }r^{p-1}(|D_{L}(r\phi)|^2+|\D(r\phi)|^2+|r\tilde{\rho}|^2+|r\a|^2+|r\si|^2)dvd\om ds\\
 +&\int_{S_{\tau}}r^{p}(|D_{L}(r\phi)|^2+|r\a|^2)dvd\om \leq CE_0(1+|\tau|)^{p-1-\ga},\quad \forall\tau\in\mathbb{R};
\end{align*}
for some constant $C$ depending only on $\ga$, $\ga_0$, $\ep$, $R$, $p$ and the charge $q_0$. Here $\tilde{F}$ is the chargeless field $\tilde{F}=F-q_0 r^{-2}\chi_{\{r\geq t+R\}}dt\wedge dr$.
\end{itemize}
\end{thm}
We make several remarks:
\begin{remark}
The weighted energy space defined in line \eqref{IDMKG} with $\ga_0=0$ is scale invariant. Thus our assumptions on the initial data can be almost scale invariant. We also note that the charge part
 $q_0 r^{-2}$ does not belong to any such weighted energy space with $\ga_0\geq 0$. This in particular implies that the existence of nonzero charge plays an important role in the asymptotics of
the solution at least in the neighborhood of the spatial infinity.
\end{remark}

\begin{remark}
Our approach in this paper can be adapted to obtain similar decay estimates for solutions of \eqref{EQMKG} in higher dimensions as long as the solution exists globally. The problem in higher dimensions is that
global regularity in energy space is critical in $4+1$-dimension and super-critical in higher dimensions. The global regularity in energy space is known in the critical case with small
energy \cite{MKGtataru}. Recently the smallness assumption has been removed by Oh-Tataru \cite{OhMKG4} and  Krieger-Luhrmann\cite{KriegerMKG4}.
\end{remark}

\begin{remark}
 We can relax the assumptions on the components of the initial data $D_{\Lb}\phi$, $\underline{\a}$ to be merely
bounded in the energy space instead of belonging to the weighted energy space. This is because the decay estimates in the interior region rely only on the energy flux through $\Si_0$
and a $r$-weighted energy flux (see definition in Section \ref{defpwe}) through the forward light cone $\{t=r+R\}$. The latter one is independent of the components $D_{\Lb}\phi$, $\underline{\a}$.
\end{remark}

\begin{remark}
The regularity in the exterior region is propagated. If initially the components $D_{\Lb}\phi$, $\underline{\a}$ are merely bounded in the energy space, then we can only obtain the
boundedness of the energy in the exterior region instead of quantitative decay as in the main theorem. If the initial weighted energy $E_0$ defined in line \eqref{IDMKG} is finite for some $\ga_0>1$, then we can show that the
integrated local energy and the energy flux through $\Si_{\tau}$ decay with the maximal rates $(1+\tau)^{-2}$ in the interior region $\{t+R\geq r\}$ while in the exterior region the decay estimates are the same
as in the main theorem with all $\ga<\ga_0$. This propagation of regularity in the exterior region has been discussed in \cite{MKGLydia}, \cite{LindbladMKG} for the small data case.
\end{remark}

\begin{remark}
 For the non-abelian case of Yang-Mills equations, we can obtain the similar decay estimates when the charges are vanishing. For the general case, the definitions of charges that we are aware of are not gauge
invariant. In particular, we may need to work under a fixed gauge condition if we expect similar decay estimates as in the main theorem. We will address this issue in our future works.
\end{remark}

 As far as we know, Theorem \ref{energydecayMKG} is the first result that gives the strong decay estimates for the global large solutions of \eqref{EQMKG}. The quantitative decay estimates rely only the weighted energy norm of the initial data. Based on our experience on nonlinear wave equations, these decay estimates, more precisely, the energy flux decay through $\Si_{\tau}$, the integrated local energy estimate together with the hierarchy of $r$-weighted energy estimates, are sufficiently strong to  derive the pointwise decay of the solutions if we also have the decay estimates for the derivatives of the solutions, see e.g. \cite{yang1}, \cite{yang3}. However, the MKG equation is nonlinear. Commuting with derivatives will introduce nonlinear terms. The quantitative decay estimates in Theorem \ref{energydecayMKG} are  fundamental to control those nonlinear terms without assuming any smallness. This will be addressed in our future works.

\bigskip

As we have explained previously, the conformal compactification method may not be able to derive the decay estimates in Theorem \ref{energydecayMKG} due to the existence of nonzero charge and the weak decay of the initial data. The perturbation method works only for small data case, at least some form of smallness is necessary, see e.g. \cite{yang5} in which the incoming energy flux was assumed to be small. In this paper, we use a new approach to study the large data problem to \eqref{EQMKG}. This new approach was originally introduced by Dafermos-Rodnianski in \cite{newapp} for the study of decay of linear waves on black hole spacetimes and has been successfully applied to several linear and nonlinear problems, e.g. \cite{yannisER}, \cite{newapp3}, \cite{gustavUSchwarz}, \cite{volkerSch}, \cite{qiansharpLocal}, \cite{yang2}. However these applications were carried out in a linear setting as either the problem itself is linear or the data are small so that perturbation method was applied. Theorem \ref{energydecayMKG} is the first result showing that this robust new approach is also very useful for some nonlinear large data problems.

\bigskip

In the abstract framework proposed by Dafermos-Rodnianski in \cite{newapp}, the new method for proving the decay estimates as in Theorem \ref{energydecayMKG} relies on three kinds of basic ingredients and estimates: a uniform energy bound, an integrated local energy decay estimate and a hierarchy of $r$-weighted energy estimates in a neighbourhood of the null infinity. The uniform energy bound usually follows from the energy conservation law or it can be derived by using the vector field $\pa_t$ as multiplier. The integrated local energy decay estimate has been well studied. This type of estimates was first proven by C. Morawetz in \cite{mora1}, \cite{mora2}. In the past decades, the integrated local energy decay estimate has been obtained for variant linear waves on various kinds of backgrounds, see e.g. \cite{bluemaxkerr}, \cite{dr3}, \cite{tatarumaxile}, \cite{lodecTataru} and references therein. The method may vary for different problems. For solutions of \eqref{EQMKG}, we can use the vector field $f(r)\pa_r$ as multiplier to derive the integrated local energy decay estimate.

\bigskip

The very important new ingredient of the new method is the $r$-weighted energy estimates in a neighborhood of the null infinity. This kind of estimates can be obtained by using the vector fields $r^p(\pa_t+\pa_r)$ as multipliers for all $0\leq p\leq 2$. Combining these robust new estimates with the integrated local energy decay estimates discussed above, a pigeon hold argument then leads to the quantitative decay of the energy flux through $\Si_{\tau}$. In other words, to prove the decay estimates as in Theorem \ref{energydecayMKG}, these three kinds of estimates have to be considered together.

\bigskip

The MKG equations are nonlinear while the new method discussed above is under a linear setting. One of the key observations that allows us to use the new method to obtain the decay estimates in Theorem \ref{energydecayMKG} is that the total energy momentum tensor (see definitions in Section \ref{notation}) for the full solution $(F, \phi)$ of \eqref{EQMKG} is divergence free. Based on this fact, we can use the vector fields $\pa_t$, $f(r)\pa_r$, $r^p(\pa_t+\pa_r)$ as multipliers to derive those three kinds of basic estimates for the new approach.

\bigskip

However, due to the presence of nonzero charge, we are not able to get useful $r$-weighted energy estimates for any $p\geq 1$ in the exterior region. This is because the charge part $q_0 r^{-2}$ does not belong to any weighted energy space with weights $r^{p}$, $p\geq 1$. On the other hand, it seems that to close those basic estimates, at lease some $r$-weighted energy estimate with $p>1$ is necessary. To overcome this difficulty caused by the nonzero charge, we define the charge 2-form $\bar F$
\[
 \bar F=q_0 r^{-2}\chi_{\{t+R\leq r\}}dt\wedge dr
\]
in the exterior region, which has the same charge $q_0$ as the full solution $F$. By our assumption, the chargeless field $F-\bar F$ is bounded in the weighted energy space with weights $r^{1+\ga_0}$ initially. We thus can apply the vector field method to the energy momentum tensor for the chargeless part of the solution $(F-\bar F, \phi)$ instead of the full solution $(F, \phi)$. However the energy momentum tensor for $(F-\bar F, \phi)$ is not divergence free. It will introduce the following error term:
\[
 q_0 \Im(\overline{D\phi}\cdot\phi)
\]
arising from the interaction between the charge and the scalar field $\phi$. Since the charge $q_0$ is large, this error term can not be absorbed in general. The key to control this error term is to make use
of the better decay of the scalar field $\phi$ initially and then use a type of Gronwall's inequality, for details see Section \ref{secex}. This is the reason that we loss a little bit of decay rate ($\ga<\ga_0$) in the main theorem.

\bigskip

Once we have the decay estimates in the exterior region, that is, the estimates in Theorem \ref{energydecayMKG} for $\tau<0$, we in particular can conclude that the energy flux and the $r$-weighted energy flux through the boundary $\Si_{0}$ of the interior region is finite. The key point that
the charge does not affect the asymptotic behavior of the solution in the interior region is that the $r$-weighted energy flux through the boundary $S_0$ depends only on $F-\bar F$. However the energy flux through $S_{0}$ does rely on the full solution $F$ but the charge part $\bar F$ has finite energy flux through $S_0$. Therefore the energy flux and the $r$-weighted energy flux through $\Si_0$ of the full solution $(F, \phi)$ are finite. This enables us to use the new approach to derive the decay estimates of Theorem \ref{energydecayMKG} in the interior region.

\bigskip

The paper is organized as follows: we derive the energy identities obtained by using the vector fields $f(r)\pa_r$, $\pa_t$, $r^p\pa_v$ as multipliers in Section \ref{notation}; we then use these identities to derive
decay estimates for the chargeless part of the solution $(F-\bar F, \phi)$ in the exterior region as well as the energy flux through the boundary $S_0$ in Section \ref{secex}; once we have the energy flux
bound through $\Si_0$ and the $r$-weighted energy flux bound through $S_0$, we can obtain the decay estimates for the solution in the interior region and conclude the main theorem in the last Section.

\textbf{Acknowledgments}
The author would like to thank Mihalis Dafermos and Pin Yu for helpful discussions. He is indebted to Igor Rodnianski for plenty of invaluable comments and suggestions.

\section{Preliminaries and energy identities}

\label{notation}

We define some additional notations used in the sequel. In the exterior region $\{r\geq R+t\}$, for $R\leq r_1< r_2$,
we use $S_{r_1, r_2}$ to denote the following outgoing null hypersurface emanating from the sphere with radius $r_1$
\[
 S_{r_1, r_2}:=\{u=-\frac{r_1}{2},\quad r_1\leq r\leq r_2\}.
\]
We note that $S_{r_1, r_2}$ is part of the outgoing null hypersurface $S_{\frac{R-r_1}{2}}$ defined before the main theorem. We use $\bar C_{r_1, r_2}$ to denote the following incoming null hypersurface emanating from the sphere with radius $r_2$
\[
 \bar C_{r_1, r_2}:=\{v=\frac{r_2}{2},\quad r_1\leq r\leq r_2\}.
\]
For $0\leq \tau_1<\tau_2$, $\bar C_{\tau_1, \tau_2}$ will be used to denote the null infinity between $\Si_{\tau_1}$ and $\Si_{\tau_2}$. Throughout this paper, $\tau\geq0$ will be the parameter
in the interior region $\{r\leq R+t\}$ and $R\leq r_1\leq r_2$ will be used as the parameters in the exterior region. On the initial hypersurface $\mathbb{R}^3$, the annulus with radii $0\leq r_1<r_2$ is denoted as
\[
 B_{r_1, r_2}:=\{t=0,\quad r_1\leq r\leq r_2\}.
\]
We use $\mathcal{D}_{r_1, r_2}$ to denote the region bounded by $S_{r_1, r_2}$, $B_{r_1, r_2}$, $\bar C_{r_1, r_2}$ and
$\mathcal{D}_{\tau_1, \tau_2}$ to denote the region bounded by $\Si_{\tau_1}$, $\Si_{\tau_2}$ for $0\leq \tau_1<\tau_2$. For the complex scalar field $\phi$ and the 2-form $F$, we denote
\[
\bar D\phi:=(D\phi, (1+r)^{-1}\phi),\quad |F|^2:=\rho^2+\si^2+\f12(|\a|^2+|\underline{\a}|^2).
\]
We now review the energy method for the MKG equations. Denote $d\vol$ the
volume form in the Minkowski space. In the local coordinate system $(t, x)$, we have
$
d\vol=dxdt.
$
Here we have chosen $t$ to be the time orientation. For any two forms $\tilde{F}$ satisfying the Bianchi identity \eqref{bianchi} and any complex scalar field $\phi$, we define the associated energy momentum tensor
\begin{equation*}
\begin{split}
 T[\phi, \tilde{F}]_{\a\b}&=\tilde{F}_{\a\mu}\tilde{F}_\b^{\;\mu}-\frac{1}{4}m_{\a\b}\tilde{F}_{\mu\nu}\tilde{F}^{\mu\nu}+\Re\left(\overline{D_\a\phi}D_\b\phi\right)-\f12 m_{\a\b}\overline{D^\mu\phi}D_\mu\phi.
\end{split}
\end{equation*}
Given a vector field $X$, we have the following identity
\[
\pa^\mu(T[\phi,\tilde{F}]_{\mu\nu}X^\nu) = \Re(\Box_A \phi X^\nu \overline{D_\nu\phi})+X^\nu F_{\nu\a}J^\a+\pa^\mu \tilde{F}_{\mu\a}\tilde{F}_{\nu}^{\;\a}X^{\nu}+T[\phi,\tilde{F}]^{\mu\nu}\pi^X_{\mu\nu},
\]
where $\pi_{\mu\nu}^X=\f12 \mathcal{L}_X m_{\mu\nu}$ is the deformation tensor of the vector field $X$ in Minkowski space and $J_\mu=\Im(\phi\cdot \overline{D_\mu\phi})$. Through out this paper, we raise and lower indices with respect to this flat metric $m_{\mu\nu}$.
Take any function $\chi$. We have the following equality
\begin{align*}
 \f12\pa^{\mu}\left(\chi \pa_\mu|\phi|^2-\pa_\mu\chi|\phi|^2\right)= \chi \overline{D_\mu\phi}D^\mu\phi -\f12\Box\chi\cdot|\phi|^2+\chi \Re(\Box_A\phi\cdot \overline\phi).
\end{align*}
We now define the vector field $\tilde{J}^X[\phi, \tilde{F}]$ with components
\begin{equation}
\label{mcurent} \tilde{J}^X_\mu[\phi,\tilde{F}]=T[\phi,\tilde{F}]_{\mu\nu}X^\nu -
\f12\pa_{\mu}\chi \cdot|\phi|^2 + \f12 \chi\pa_{\mu}|\phi|^2+Y_\mu
\end{equation}
for some vector field $Y$ which may also depend on the scalar field $\phi$. We then have the equality
\begin{align*}
\pa^\mu \tilde{J}^X_\mu[\phi,\tilde{F}] =&\Re(\Box_A \phi(\overline{D_X\phi}+\chi\overline \phi))+div(Y)+X^\nu F_{\nu\mu}J^\mu+\pa^\mu \tilde{F}_{\mu\ga}\tilde{F}_{\nu}^{\;\ga}X^{\nu}
\\&+T[\phi,\tilde{F}]^{\mu\nu}\pi^X_{\mu\nu}+
\chi \overline{D_\mu\phi}D^\mu\phi -\f12\Box\chi\cdot|\phi|^2.
\end{align*}
Here the operator $\Box$ is the wave operator in Minkowski space and the divergence of the vector field $Y$ is also taken in the Minkowski space.

\bigskip

Now take any region $\mathcal{D}$ in $\mathbb{R}^{3+1}$. Assume on this region the scalar field $\phi$ and the 2-form $\tilde{F}$ satisfies the following equation
\[
 \pa^\nu \tilde{F}_{\mu\nu}=J_{\mu},\quad \Box_A\phi=0.
\]
Here we note that the covariant operator $\Box_A$ is associated to the solution $F$ of \eqref{EQMKG}. Then using Stokes' formula, the above calculation leads to the following energy identity
\begin{align}
\notag &\iint_{\mathcal{D}} div(Y)+X^\nu (F_{\nu\mu}-\tilde{F}_{\nu\mu})J^\mu+T[\phi,\tilde{F}]^{\mu\nu}\pi^X_{\mu\nu}+
\chi \overline{D_\mu\phi}D^\mu\phi -\f12\Box\chi\cdot|\phi|^2d\vol\\
&=\iint_{\mathcal{D}}\pa^\mu \tilde{J}^X_\mu[\phi,\tilde{F}]d\vol=\int_{\pa \mathcal{D}}i_{\tilde{J}^X[\phi,\tilde{F}]}d\vol,
\label{energyeq}
\end{align}
where $\pa\mathcal{D}$ denotes the boundary of the domain $\mathcal{D}$ and $i_Z d\vol$ denotes the contraction of the volume form $d\vol$
with the vector field $Z$ which gives the surface measure of the
boundary. For example, for any basis $\{e_1, e_2,
\ldots, e_n\}$, we have $i_{e_1}( de_1\wedge de_2\wedge\ldots
de_k)=de_2\wedge de_3\wedge\ldots\wedge de_k$.

\bigskip

In this paper, the domain $\mathcal{D}$ will be regular regions bounded by the $t$-constant slices, the outgoing null hypersurfaces $u=constant$ or the incoming null
hypersurfaces $v=constant$. We now compute $i_{\tilde{J}^{X}[\phi,\tilde F]}d\vol$ on these three kinds of hypersurfaces. On
$t=constant$ slice, the surface measure is a function times $dx$. Recall the volume form
\[
d\vol=dxdt=-dtdx.
\]
Here note that $dx$ is a $3$-form. We thus can show that
\begin{equation}
\label{curlessR}
\begin{split}
 i_{\tilde{J}^{X}[\phi,\tilde{F}]}d\vol
 =&-(\Re(\overline{D^t\phi}
D_X\phi)-\f12 X^0\overline{D^\ga\phi}D_\ga\phi-\f12 \pa^t\chi \cdot
|\phi|^2+\f12\chi\pa^t|\phi|^2+Y^0\\
&+\tilde{F}^{0\mu}\tilde{F}_{\nu\mu}X^\nu-\frac{1}{4}X^0\tilde{F}_{\mu\nu}\tilde{F}^{\mu\nu})dx.
\end{split}
\end{equation}
On the outgoing null hypersurface $\{u=constant\}$, we can write the volume form
\[
d\vol=dxdt=r^2drdt d\om=2r^2dvdud\om=-2dudvd\om.
\]
Here $u=\frac{t-r}{2}$, $v=\frac{t+r}{2}$ are the null coordinates and $d\om$ is the standard surface measure on the unit sphere.
Notice that $\Lb=\pa_u$. We can compute
\begin{equation}
\label{curStau}
\begin{split}
i_{\tilde{J}^{X}[\phi, \tilde F]}d\vol=&-2(\Re(\overline{D^{\Lb}\phi}
D_X\phi)-\f12 X^{\Lb}\overline{D^\ga\phi}D_\ga\phi-\f12
\pa^{\Lb}\chi|\phi|^2+\f12\chi \pa^{\Lb}|\phi|^2+Y^{\Lb}\\
&+\tilde{F}^{\Lb\mu}\tilde{F}_{\nu\mu}X^\nu-\frac{1}{4}X^{\Lb}\tilde{F}_{\mu\nu}\tilde{F}^{\mu\nu})r^2dvd\om.
\end{split}
\end{equation}
Similarly, on the $v$-constant incoming null hypersurfaces $\{v=constant\}$, we
have
\begin{equation}
\label{curnullinfy}
\begin{split}
  i_{\tilde{J}^{X}[\phi,\tilde{F}]}d\vol=&2(\Re(\overline{D^{L}\phi}
D_X\phi)-\f12 X^{L}\overline{D^\ga\phi}D_\ga\phi-\f12
\pa^{L}\chi|\phi|^2+\f12\chi \pa^{L}|\phi|^2+Y^{L}\\
&+\tilde{F}^{L\mu}\tilde{F}_{\nu\mu}X^\nu-\frac{1}{4}X^L\tilde{F}_{\mu\nu}\tilde{F}^{\mu\nu})r^2dud\om.
\end{split}
\end{equation}
We remark here that the above formulae hold for any vector fields $X$, $Y$ and any function $\chi$.

\subsection{The $r$-weighted energy estimates using the multiplier $r^pL$}
\label{defpwe}
In this section, we establish the $r$-weighted energy identities for solutions of the MKG equations. We obtain the $r$-weighted energy estimate either in the exterior region $\{r\geq R+t\}$ for the domain $\mathcal{D}_{r_1, r_2}$ for $R\leq r_1\leq r_2$ which is bounded by the outgoing null
hypersurface $S_{\frac{R-r_1}{2}}$ emanating from the sphere with radius $r_1$,
the incoming null hypersurface emanating from the sphere with radius $r_2$ and the initial hypersurface $\mathbb{R}^3$ or in the interior region for domain $\tilde{\mathcal{D}}_{\tau_1, \tau_2}$ for
$0\leq \tau_1<\tau_2$ which is bounded by the outgoing null hypersurfaces $S_{\tau_1}$, $S_{\tau_2}$ and the cylinder $\{r=R\}$. Recall here that $R>1$ is a constant determined by the initial data and is
used to define the foliation. In the energy identity \eqref{energyeq}, we choose the vector fields $X$, $Y$ and the function $\chi$ as follows:
\[
 X=r^{p}L, \quad Y=\frac{p}{2}r^{p-2}|\phi|^2L,\quad  \chi=r^{p-1}.
\]
Note that we have the equality
\begin{align*}
 r^2|D_L\phi|^2=|D_L\psi|^2-L(r|\phi|^2)&,\quad r^2|\D\phi|^2=|\D\psi|^2,\\
r^2|D_{\Lb}\phi|^2=|D_{\Lb}\psi|^2+\Lb(r|\phi|^2)&,\quad \psi=r\phi.
\end{align*}
Throughout this paper, we always use $\psi$ to denote the weighted scalar field $r\phi$.
We then can compute
\begin{align*}
 & div(Y)+T[\phi, \tilde{F}]^{\mu\nu}\pi_{\mu\nu}^X+\chi \overline{D^{\mu}\phi}D_\mu\phi-\f12\Box\chi |\phi|^2\\
&=\frac{p}{2}r^{-2}L(r^{p}|\phi|^2)+\f12r^{p-1}\left(p(|D_L\phi|^2+|\tilde\a|^2)+(2-p)(|\D\phi|^2+\tilde{\rho}^2+\tilde\si^2)\right)\\
&\quad -\f12 p(p-1)r^{p-3}|\phi|^2+X^\nu (F_{\nu\a}-\tilde{F}_{\nu\a})J^\a\\
&=\f12 r^{p-3}\left(p(|D_L\psi|^2+|\tilde\a|^2)+(2-p)(|\D\psi|^2+\tilde{\rho}^2+\tilde\si^2)\right)+r^p (F_{L\mu}-\tilde{F}_{L\mu})J^\mu.
\end{align*}
We next compute the boundary terms according to the formulae \eqref{curlessR} to \eqref{curnullinfy}. We have
\begin{align*}
 \int_{S_{r_1, r_2}}i_{\tilde{J}^X[\phi, \tilde{F}]}d\vol&=\int_{S_{r_1, r_2}}r^{p}(|D_L\psi|^2+r^2|\tilde\a|^2)-\f12 L(r^{p+1}\phi) \quad dvd\om,\\
\int_{\bar C_{r_1, r_2}}i_{\tilde{J}^X[\phi, \tilde{F}]}d\vol&=-\int_{\bar C_{r_1, r_2}}r^{p}(|\D\psi|^2+r^2|\tilde{\rho}|^2+r^2|\tilde\si|^2)+\f12\Lb(r^{p+1}|\phi|^2) dud\om,\\
\int_{B_{r_1, r_2}}i_{\tilde{J}^X[\phi,\tilde{F}]}d\vol&=\f12\int_{B_{r_1, r_2}}r^p(|D_L\psi|^2+|\D\psi|^2)+r^{p+2}(|\tilde\a|^2+|\tilde{\rho}|^2+\tilde\si^2)\\
&\quad\quad-\pa_r(r^{p+1}|\phi|^2) d\om dr,\\
\int_{r=R, \tau_1\leq t\leq \tau_2 }i_{\tilde{J}^X[\phi, \tilde{F}]}d\vol&=\f12\int_{\tau_1}^{\tau_2}\int_{\om}r^p(|D_L\psi|^2-|\D\psi|^2+r^2(|\tilde \a|^2-|\tilde\rho|^2-|\tilde\si|^2))\\
&\quad\quad-\pa_t(r^{p+1}|\phi|^2) d\om dt.
\end{align*}
The formula on $S_{\tau}$ is the same as that on $S_{r_1, r_2}$. Now notice that on the domain $D_{r_1, r_2}$ in the exterior region we have the identity
\begin{align*}
 &\int_{S_{r_1, r_2}}L(r^{p+1}|\phi|^2)dvd\om-\int_{\bar C_{r_1, r_2}}\Lb(r^{p+1}|\phi|^2)dud\om-\int_{B_{r_1, r_2}}\pa_r(r^{p+1}|\phi|^2)d\om dr=0
\end{align*}
and on the domain $\tilde{D}_{\tau_1, \tau_2}$ in the interior region we have
\begin{align*}
 &-\int_{S_{\tau_1}}L(r^{p+1}|\phi|^2)dvd\om-\int_{\bar C_{\tau_1, \tau_2}}\Lb(r^{p+1}|\phi|^2)dud\om\\
&+\int_{S_{\tau_2}}L(r^{p+1}|\phi|^2)dvd\om+\int_{\tau_1}^{\tau_2}\int_{\om}\pa_t(r^{p+1}|\phi|^2)d\om dt=0.
\end{align*}
The above calculations then lead to the following $r$-weighted energy identity in the exterior region
\begin{align}
\notag
  &\iint_{\mathcal{D}_{r_1, r_2}}r^{p-1}\left(p(|D_L\psi|^2+r^2|\tilde\a|^2)+(2-p)(|\D\psi|^2+r^2|\tilde{\rho}|^2+r^2\tilde\si^2)\right)dvd\om du\\
\label{pWescaout}
&+\int_{S_{r_1,r_2}}r^p(|D_L\psi|^2+r^2|\tilde\a|^2)dvd\om+\int_{\bar C_{r_1, r_2}}r^p(|\D\psi|^2+r^2|\tilde{\rho}|^2+r^2\tilde\si^2)dud\om\\
\notag
=&\f12\int_{B_{r_1,r_2}}r^p(|D_L\psi|^2+|\D\psi|^2)+r^{p+2}(|\tilde\a|^2+|\tilde{\rho}|^2+\tilde\si^2)drd\om-\iint_{\mathcal{D}_{r_1, r_2}}r^{p}(F_{L\mu}-\tilde{F}_{L\mu})J^\mu dxdt
\end{align}
and the corresponding one in the interior region
\begin{equation}
\label{pWescain}
\begin{split}
&\int_{\tau_1}^{\tau_2}\int_{S_\tau}r^{p-1}\left(p(|D_L\psi|^2+r^2|\tilde\a|^2)+(2-p)(|\D\psi|^2+r^2|\tilde{\rho}|^2+r^2\tilde\si^2)\right)dvd\om du\\
&+\int_{S_{\tau_1}}r^p(|D_L\psi|^2+r^2|\tilde\a|^2)dvd\om+\int_{\bar C_{\tau_1, \tau_2}}r^p(|\D\psi|^2+r^2|\tilde{\rho}|^2+r^2\tilde\si^2)dud\om\\
=&\int_{S_{\tau_1}}r^p|D_L\psi|^2+r^{p+2}|\tilde{\a}|^2dvd\om-\f12\int_{\tau_1}^{\tau_2}\int_{\om}r^p(|D_L\psi|^2-|\D\psi|^2+r^2(|\tilde{\a}|^2-|\tilde\rho|^2-\tilde\si^2)) d\om dt\\
&-\int_{\tau_1}^{\tau_2}\int_{S_\tau}r^{p}(F_{L\mu}-\tilde{F}_{L\mu})J^\mu dxdt.
\end{split}
\end{equation}
Here we recall that
\[
 J_\mu=\Im(\overline{D_\mu\phi}\cdot \phi)=r^{-2}\Im(\overline{D_\mu\psi}\cdot \psi),\quad \psi=r\phi.
\]
In the exterior region, we consider estimates for the chargeless part $\tilde{F}$ while in the interior region, we take $\tilde{F}$ as the full solution $F$. The above $r$-weighted energy identity
explains why the charge can only affect the asymptotic behaviour of the solution in the exterior region. In fact, we see later that the charge only affect the asymptotics for the curvature components
$\rho$. The solution on the interior region only depends on the data on $\Si_0$. More precisely, the energy flux through it and the $r$-weighted energy flux. However, the $r$-weighted energy flux depends only on
the curvature component $\a$ but not $\rho$. This shows that the charge can only affect the asymptotics for the solution on the exterior region. This phenomenon was original conjectured by W. Shu in \cite{Shu}.

\subsection{The integrated local energy estimates using the multiplier $f(r)\pa_r$}
We consider the integrated local energy estimates either on the exterior region for the domain $\mathcal{D}_{r_1, r_2}$, $R\leq r_1\leq r_2$ or on the interior region for the domain
$\mathcal{D}_{\tau_1, \tau_2}$, $0\leq \tau_1<\tau_2$ which is bounded by the hypersurfaces $\Si_{\tau_1}$ and $\Si_{\tau_2}$. In the energy identity \eqref{energyeq}, we choose the vector fields
$X$, $Y$ and the function $\chi$ as follows
\[
 X=f(r)\pa_r, \quad \chi=r^{-1}f, \quad Y=0
\]
for some function $f$ of $r$. We then can compute
\begin{align*}
 &T[\phi, \tilde{F}]^{\mu\nu}\pi^X_{\mu\nu}+ \chi \overline{D_\mu\phi}D^\mu\phi-\f12\Box \chi |\phi|^2\\
=&\f12 f'(|D_t\phi|^2+|D_r\phi|^2+\f12 |\tilde\a|^2+\f12 |\underline{\tilde{\a}}|^2)+(r^{-1}f-\f12 f')(|\D\phi|^2+\tilde\rho^2+\tilde\si^2)-\f12 r^{-1}\pa_r f'|\phi|^2.
\end{align*}
We will construct the function $f$ explicitly later as the choice of the function depends on the region we consider. The basic idea is to choose function $f$
such that $f'$, $r^{-1}f-\f12 f'$, $-\pa_r f'$ are all positive. We now compute the boundary terms according to the formulae \eqref{curlessR} to \eqref{curnullinfy}. We can show that
\begin{align*}
 \int_{S_{r_1, r_2}}i_{\tilde{J}^X[\phi, \tilde{F}]}d\vol&=\f12\int_{S_{r_1, r_2}}f(|D_L\phi|^2-|\D\phi|^2-\tilde{\rho}^2+|\tilde\a|^2-|\tilde\si|^2)-\chi'|\phi|^2\\
&\quad \quad+2\chi \Re(D_L\phi \cdot \overline\phi) \quad r^2dvd\om,\\
\int_{\bar C_{r_1, r_2}}i_{\tilde{J}^X[\phi, \tilde{F}]}d\vol&=\f12\int_{\bar C_{r_1, r_2}}f(|D_{\Lb}\phi|^2-|\D\phi|^2-|\tilde\rho|^2+|\underline{\tilde{\a}}|^2-|\tilde\si|^2)
-\chi'|\phi|^2\\
&\quad \quad-2\chi \Re(D_{\Lb}\phi\cdot \overline{\phi}) \quad r^2dud\om,\\
\int_{B_{r_1, r_2}}i_{\tilde{J}^X[\phi,\tilde{F}]}d\vol&=\int_{B_{r_1, r_2}}f(\Re(\overline{D_t\phi}(D_r\phi+r^{-1}\phi))+\frac{1}{4}f(|\tilde\a|^2-|\tilde{\underline{\a}}|^2) dx.
\end{align*}
Here $\chi=r^{-1}f$. The idea is that we use the energy flux through the corresponding surfaces to bound these boundary terms. In fact, we can compute the energy flux explicitly
\begin{align}
 \notag
 E[\phi, \tilde{F}](S_{r_1, r_2}):=&2\int_{S_{r_1, r_2}}i_{\tilde{J}^{\pa_t}[\phi, \tilde{F}]}d\vol=\int_{S_{r_1, r_2}}(|D_L\phi|^2+|\D\phi|^2+\tilde{\rho}^2+|\tilde\a|^2+|\tilde\si|^2) r^2dvd\om,\\
\notag
E[\phi, \tilde{F}](\bar C_{r_1, r_2}):=&2\int_{\bar C_{r_1, r_2}}i_{\tilde{J}^{\pa_t}[\phi, \tilde{F}]}d\vol=\int_{\bar C_{r_1, r_2}}(|D_{\Lb}\phi|^2+|\D\phi|^2+|\tilde\rho|^2+|\underline{\tilde{\a}}|^2+|\tilde\si|^2) r^2dud\om,\\
\label{enfluxc}
E[\phi, \tilde{F}](B_{r_1, r_2}):=&2\int_{B_{r_1, r_2}}i_{\tilde{J}^{\pa_t}[\phi,\tilde{F}]}d\vol=\int_{B_{r_1, r_2}}|D\phi|^2+|\tilde\rho|^2+|\tilde\si|^2+\f12(|\tilde\a|^2+|\tilde{\underline{\a}}|^2) dx.
\end{align}
Therefore to control the boundary terms for the multiplier $f(r)\pa_r$, it suffices to control the integral of $|\phi|^2$ in terms of the corresponding energy flux.
We will choose the function $f$ to be bounded. Since $\chi=r^{-1}f$, we have $\chi'\sim r^{-2}$. We thus can use a type of Hardy's inequality adapted to the connection $D$ to bound the integral of $\chi'|\phi|^2$.
\begin{lem}
 \label{Hardy}
Assume the complex scalar field $\phi$ vanishes at null infinity, that is,
\[
 \lim_{v\rightarrow\infty}\phi(v, u, \om)=0
\]
for all $u$, $\om$. Then we have
\begin{align*}
 \int_{\Si_{\tau}}\left|\frac{\phi}{1+r}\right|^2d\si\leq  12 E[\phi](\Si_{\tau}),\quad \int_{S_{r_1}}\left|\frac{\phi}{1+r}\right|^2r^2dvd\om\leq 12 E[\phi](S_{r_1}),\quad r_1>R,\quad \tau\geq 0.
\end{align*}
Here $d\si=dx$ when $r\leq R$ and $d\si=r^2dvd\om $ on $S_{\tau}$ and $E[\phi](\Si)$ denotes the energy flux through the surface $\Si$.
\end{lem}
\begin{proof}
It suffices to notice that the covariant derivative $D$ is compatible with the inner product $<,>$ on the complex plane. Then the proof goes the same as the case when the connection field $A$ is trivial,
see e.g. \cite{dr3}, \cite{yang1}. Another quick way to reduce the proof of the Lemma to the case with trivial connection field $A$ is to choose a particular gauge such that the scalar field
$\phi$ is real. We can do this is due to the fact that all the quantities we considered here are gauge invariant.
\end{proof}
The above lemma implies that the boundary terms arising from the multiplier $f(r)\pa_r$ can be controlled by the corresponding energy flux up to a constant. We now choose $f$ explicitly to establish the
integrated local energy estimates. On the interior region $\mathcal{D}_{\tau_1, \tau_2}$ we choose
$$f=2\ep^{-1}-\frac{2\ep^{-1}}{(1+r)^{\ep}},\quad \chi=r^{-1}f$$
and on the exterior region $\mathcal{D}_{r_1, r_2}$ we take
\[
 f=2\ep^{-1}(r_1^{-\ep}-(1+r)^{-\ep}),\quad r_1\geq R>1
\]
for some small positive constant $\ep<\frac{1}{4}$. In any case, $f$ is bounded and we have
\begin{align*}
 \f12 f'=\frac{1}{(1+r)^{1+\ep}},\quad -\f12 \Box \chi&=\frac{1+\ep}{r(1+r)^{2+\ep}},\quad \chi- \f12
f'\geq \frac{2\ep^{-1}}{r} - \frac{1+2\ep^{-1}}{r(1+r)^\ep}\geq \frac{1}{r}.
\end{align*}
The last inequality holds for $r>1$.

We still need to estimate the error term $X^\nu(F_{\nu\mu}-\tilde{F}_{\nu\mu})J^{\mu}$. Recall that
\[
 J^\mu=\Im(\phi\cdot\overline{D^\mu\phi}).
\]
Using Cauchy-Schwarz's inequality, we have
\[
 2|X^\nu(F_{\nu\mu}-\tilde{F}_{\nu\mu})J^{\mu}|\leq f(\ep_1^{-1}(1+r)^{1+\ep}|F_{r\mu}-\tilde{F}_{r\mu}|^2|\phi|^2+\ep_1(1+r)^{-1-\ep}|D_\mu\phi|^2),\quad \forall \ep_1>0.
\]
The idea to control this error term is that we choose $\ep_1$ sufficiently small so that the second term on the right hand side can be absorbed.

\bigskip

To avoid too many constants, in the rest of this paper, we make a convention that $A\les B$ means $A\leq C B$ for some constant $C$ depends only on
the constants $\ep$, $R$, the charge $q_0$ and $\ga$, $\ga_0$ in the main Theorem \ref{energydecayMKG}.

\bigskip

Based on the above calculations, we can derive the following integrated local energy estimates in the
interior region for the domain $\mathcal{D}_{\tau_1, \tau_2}$, $0\leq \tau_1<\tau_2$
\begin{align}
\label{ILEMKGin}
 &\int_{\tau_1}^{\tau_2}\int_{\Si_{\tau}}\frac{|\bar D\phi|^2+|\tilde{F}|^2}{(1+r)^{1+\ep}}+\frac{|\D\phi|^2+\tilde\rho^2+\tilde\si^2}{1+r}dxdt\\
\notag
\les & E[\phi, \tilde{F}](\Si_{\tau_1})+E[\phi, \tilde F](\Si_{\tau_2})
+E[\phi, \tilde{F}](\bar C_{\tau_1, \tau_2})\\
\notag
&\quad+\int_{\tau_1}^{\tau_2}\int_{\Si_{\tau}}(1+r)^{1+\ep}|F_{r\nu}-\tilde{F}_{r\nu}|^2|\phi|^2dxdt
\end{align}
and the integrated local energy estimates in the exterior region for the domain $\mathcal{D}_{r_1, r_2}$, $R\leq r_1<r_2$
\begin{align}
\label{ILEMKGout}
 &r_1^\ep\iint_{\mathcal{D}_{r_1, r_2}}\frac{|\bar D\phi|^2+|\tilde{F}|^2}{(1+r)^{1+\ep}}+\frac{|\D\phi|^2+\tilde\rho^2+\tilde\si^2}{1+r}dxdt\\
\notag
\les & E[\phi, \tilde{F}](B_{r_1, \infty})+E[\phi, \tilde F](S_{r_1, r_2})
+E[\phi, \tilde{F}](\bar C_{r_1, r_2})\\
\notag
&\quad+r_1^{-2\ep}\iint_{\mathcal{D}_{r_1, r_2}}(1+r)^{1+\ep}|F_{r\nu}-\tilde{F}_{r\nu}|^2|\phi|^2dxdt.
\end{align}
For the estimate in the exterior region, we gain the decay $r_1^\ep$ is due to the fact that the function $f$ has upper bound $2\ep^{-1}r_1^{-\ep}$ while in the interior region $|f|\les 1$. In addition, since
we need to use the Hardy's inequality to control the integral of $|\phi|^2$, we rely on the fact that initially $\phi$ vanishes at the spatial infinity. This is the reason that we use
$E[\phi, \tilde{F}](B_{r_1, \infty})$ instead of $E[\phi, \tilde{F}](B_{r_1, r_2})$ in the estimate.

\subsection{Energy estimates using the multiplier $\pa_t$}
Our new approach is based on the observation that all the energy should go through the null infinity. Hence the energy flux through the hypersurface $\Si_{\tau}$ which is far away from the light cone has to decay to
zero. The first step is to show that, by using the classical energy method, the energy flux is monotonic.

\bigskip

In the energy identity \eqref{energyeq}, take $X=\pa_t$, $Y=0$, $\chi=0$. Since $\pa_t$ is killing, based on the calculations \eqref{enfluxc} in the previous section, we obtain the energy estimate
in the interior region
\begin{equation}
 \label{esin}
E[\phi, \tilde{F}](\Si_{\tau_2})+E[\phi, \tilde{F}](\bar C_{\tau_1, \tau_2})\leq E[\phi, \tilde{F}](\tau_1)+2\int_{\tau_1}^{\tau_2}\int_{\Si_{\tau}}|F_{0\mu}-\tilde{F}_{0\mu}||J^\mu| d\vol
\end{equation}
and the energy estimate on the exterior region
\begin{equation}
 \label{esout}
E[\phi, \tilde{F}](S_{r_1, r_2})+E[\phi, \tilde{F}](\bar C_{r_1, r_2})\leq E[\phi, \tilde{F}](B_{r_1, r_2})+2\iint_{\mathcal{D}_{r_1, r_2}}|F_{0\mu}-\tilde{F}_{0\mu}||J^\mu| d\vol
\end{equation}
for all $0\leq \tau_1<\tau_2$ and $R\leq r_1<r_2$.

\bigskip

When coupled to the integrated local energy estimates derived in the previous section, we can estimate the last terms containing $J^\mu$ in the above energy estimates by using Cauchy Schwarz's inequality so that
on the right hand side of the inequality there is only integral of $|\phi|^2$ as in the integrated local energy estimates \eqref{ILEMKGin}, \eqref{ILEMKGout}.

\section{Energy estimates in the exterior region}
\label{secex}
Due to the presence of non-zero charge $q_0$, the component of the curvature $\rho$ has a tail $q_0 r^{-2}$. In general, a useful integrated local energy decay estimate for the
full solution $(F, \phi)$ of \eqref{EQMKG} at least in the exterior region is not expected as $\rho$ dost not decay in time $t$. A natural way to circumvent this problem is to remove the
charge part of the field and consider estimates for the remained part which is chargeless. As in the work of Lindblad-Sterbenz in \cite{LindbladMKG}, we define the
charged 2-form on the exterior region $\{r\geq R+t\}$
\begin{equation*}
\bar F=q_0 d(r^{-1}dt)=q_0 r^{-2}dt\wedge dr.
\end{equation*}
The corresponding null decomposition under the null frame $\{L, \Lb, e_1, e_2\}$ is as follows:
\begin{equation}
 \label{charF}
\bar \rho=q_0 r^{-2},\quad \bar \a=\underline{\bar \a}=0,\quad \bar \si=0.
\end{equation}
It can be checked that this charged 2-form $\bar F$ satisfies the linear Maxwell equation
\[
 \pa^\mu \bar F_{\mu\nu}=0
\]
in the exterior region $r\geq R+t$. Here we recall that the constant $R$ is a fixed constant determined by the initial data and is used to define the foliation $\Si_{\tau}$. For solution $(F, \phi)$
of \eqref{EQMKG}, we then define the chargeless 2-form
\begin{equation*}
 \tilde{F}=F-\bar F.
\end{equation*}
In this section, we do the estimates for the chargeless part $\tilde{F}$ instead of the full solution $F$.
\begin{remark}
We remark here that since we
do the estimates separately on the exterior region $\{r\geq R+t\}$ and the interior region $\{r\leq R+t\}$, we do not have to take a smooth cut off function as in \cite{LindbladMKG}.
\end{remark}

Notice that in the integrated local energy estimates \eqref{ILEMKGout}
\[
 (1+r)^{1+\ep}|F_{r\nu}-\tilde{F}_{r\nu}|^2=(1+r)^{1+\ep}|\bar F_{r\nu}|^2\sim q_0^2 r^{-3+\ep}.
\]
The decay rate is not sufficient to make that term be absorbed. Moreover, since the charge $q_0$ is large, we are even lack of the smallness needed. This means that we need to use other estimates in order to
control the error term arising from the charge part. This forces us to consider the $r$-weighted energy estimate on the exterior region $\{r\geq R+t\}$ first.

\bigskip

Take $\tilde{F}=F-\bar F$ in the $r$-weighted energy
estimate \eqref{pWescaout} in the exterior region for the charged 2-form $\bar F$ defined above. According to the relation \eqref{charF}, we obtain the following $r$-weighted energy estimate for the chargeless
part of the solution
\begin{equation}
\label{pWeMKGout}
\begin{split}
  &\iint_{\mathcal{D}_{r_1, r_2}}r^{p-1}\left(p(|D_L\psi|^2+r^2|\a|^2)+(2-p)(|\D\psi|^2+r^2|\tilde{\rho}|^2+r^2\si^2)\right)dvd\om du\\
&+\int_{S_{r_1,r_2}}r^p(|D_L\psi|^2+r^2|\a|^2)dvd\om+\int_{\bar C_{r_1, r_2}}r^p(|\D\psi|^2+r^2|\tilde{\rho}|^2+r^2\si^2)dud\om\\
=&\f12\int_{B_{r_1,r_2}}r^p(|D_L\psi|^2+|\D\psi|^2)+r^{p+2}(|\a|^2+|\tilde{\rho}|^2+\si^2)drd\om-q_0\iint_{\mathcal{D}_{r_1, r_2}}r^{p-2}J_L dxdt.
\end{split}
\end{equation}
Here we recall that
\[
 J_L=\Im(\overline{D_L\phi}\cdot \phi)=r^{-2}\Im(\overline{D_L\psi}\cdot \psi),\quad \psi=r\phi.
\]
The integral on the initial hypersurface $B_{r_1, r_2}$ is finite for all $p\leq 1+\ga_0$ by the assumption \eqref{IDMKG}. We will get a useful $r$-weighted energy inequality once we can control
the error term arising from the charge part which involves estimates for the above $J_L$ depending only on the scalar field. In general $r^{-1}\psi$ has the same size as $D_L\psi$. So if the charge $q_0$
is sufficiently small, independent of the initial data, we then can absorb the error term which has no positive sign in the above inequality \eqref{pWescaout}. However, in our setting, the charge $q_0$
is arbitrarily large. Then a possible approach to control this term is to use Gronwall's inequality. The problem is that as having explained the error term has the same decay rate with those terms on the left
hand side of the above energy equality. In other words, we will get a logarithmic growth of the error term instead of boundedness when using Gronwall's inequality.

\bigskip

In order to overcome this potential logarithmic growth, we make use of the better decay of the initial data. More precisely, we are not going to pursue a $r$-weighted energy estimates with the largest possible $p$
value which is $1+\ga_0$ by the initial condition. We instead consider the $r$-weighted energy estimates with the greatest $p$ which is slightly less than $1+\ga_0$. Let
\[
 \ep_0=\f12(\ga_0-\ga),
\]
where $\ga<\ga_0$ is the positive constant in the main Theorem \ref{energydecayMKG}. First using Cauchy-Schwarz's inequality, we can bound
\begin{equation}
\label{csjl}
 |q_0\iint_{\mathcal{D}_{r_1, r_2}}r^{p-2}J_L dxdt|\leq \frac{p}{2}\iint_{\mathcal{D}_{r_1, r_2}}r^{p-1}|D_L\psi|^2dudvd\om+\frac{8q_0^2}{p}\iint_{\mathcal{D}_{r_1, r_2}}r^{p-3}|\psi|^2dudvd\om.
\end{equation}
The first term on the right hand side can be absorbed in the previous estimate \eqref{pWescaout}. To estimate the weighted integral of $|\psi|^2$, we first note
\begin{equation}
\label{eqfact}
 \pa_L|\psi|\leq |D_L\psi|.
\end{equation}
This is because these are gauge invariant. We can simply assume $\psi$ is real. Then on the fixed outgoing null hypersurface $S_{\tau}$, we can show that
\begin{equation}
 \label{Sobpsi}
\begin{split}
\int_{\om}|\psi|^2(u, v, \om)d\om&\les \int_{\om}|\psi|^2(u, v_1, \om)d\om+\int_{\om}(\int_{v_1}^v|D_L\psi|dv)^2d\om\\
&\les\int_{\om}|\psi|^2(u, v_1, \om)d\om+r_1^{-\ga}\int_{v_1}^v\int_{\om}r^{1+\ga}|D_L\psi|^2dvd\om \cdot
\end{split}
\end{equation}
Here $r=\frac{v-u}{2}$, $\tau=\frac{R-r}{2}$ and $v_1<v$. The implicit constant depends only on $\ga$. For $p\leq 1+\ga$, multiply the above inequality by $r^{p-3}$ and then
integrate it from $v_1$ to $v$ on $S_\tau$. We obtain the following estimate:
\begin{equation}
 \label{SovpsiI}
\begin{split}
&\int_{v_1}^v\int_{\om}r^{p-3}|\psi|^2(u, v, \om)d\om dv\\
&\les r_1^{p-2}\int_{\om}|\psi|^2(u, v_1, \om)d\om+r_1^{-\ga+p-2}\int_{v_1}^v\int_{\om}r^{1+\ga}|D_L\psi|^2dvd\om \cdot
\end{split}
\end{equation}
The improved estimate for the integral of $|\psi|^2$ comes from the first term on the right hand side of the above estimate as we can choose $v_1=-u$ which means that the point $(u, v_1, \om)$ is on
the initial hypersurface. Denote
\begin{equation*}
 E_{r_1, r_2}^p=\f12\int_{B_{r_1,r_2}}r^p(|D_L\psi|^2+|\D\psi|^2+r^{-2}|\psi|^2)+r^{p+2}(|\a|^2+|\tilde{\rho}|^2+\si^2)drd\om.
\end{equation*}
Our assumptions on the initial data implies that $E_{r_1, r_2}^{1+\ga_0}$ is finite. Let $p=1+\ga$ and $v_1=-u$ in \eqref{SovpsiI}. From the energy identity \eqref{pWescaout}, the Cauchy-Schwarz's inequality
\eqref{csjl} and the above Sobolev embedding \eqref{SovpsiI}, we can conclude that
\begin{equation}
\label{psigron0}
 \int_{S_{r_1, r_2}}r^{1+\ga}|D_L\psi|^2dvd\om\leq C_1 E_{r_1, r_2}^{1+\ga}+C_1\int_{r_1}^{r_2}s^{-1}\int_{S_{s, r_2}}r^{1+\ga}|D_L\psi|^2dvd\om ds
\end{equation}
for some constant $C_1$ depending only on $\ga$ and the charge $q_0$. Without loss of generality, we can assume $C_1>1$.
Since $s^{-1}$ is not integrable, Gronwall's inequality will lead to a logarithmic growth. We fix $r_2$ and let
$r_1$ be the variable. Denote
\[
 G(r_1)=\int_{r_1}^{r_2}s^{-1}\int_{S_{s, r_2}}r^{1+\ga}|D_L\psi|^2dvd\om ds.
\]
Then we have
\[
 C_1 E_{r_1, r_2}^{1+\ga}+C_1 G(r_1)+r_1 G'(r_1)\geq 0, \quad G(r_2)=0,\quad G'(r_2)=0,\quad r_1\leq r_2.
\]
Multiply the above inequality by $r_1^{C_1-1}$ and then integrate it from $r_1$ to $r_2$. We obtain
\begin{equation}
\label{psigron1}
 G(r_1)\leq C_1 r_1^{-C_1}\int_{r_1}^{r_2}s^{C_1-1}E_{s, r_2}^\ga ds\leq r_1^{-C_1} E_{r_1, r_2}^{C_1+1+\ga}\leq r_1^{-C_1}r_2^{C_1}E_{r_1, r_2}^{1+\ga}\leq E_0 r_1^{-C_1-2\ep_0}r_2^{C_1}.
\end{equation}
Here recall that $E_0=E_{R,\infty}^{1+\ga_0}$ and $2\ep_0=\ga_0-\ga$. Then from the previous two estimates, we conclude that
\begin{equation}
 \label{lowerscal}
\int_{S_{r_1, r_2}}r^{1+\ga}|D_L\psi|^2dvd\om\leq 2E_0 r_1^{-C_1-2\ep_0}r_1^{C_1+\ep_0}=2E_0 r_1^{-\ep_0},\quad \forall r_1\leq r_2\leq r_1^*= r_1^{1+\frac{\ep_0}{C_1}}.
\end{equation}
Here by our definition, $C_1$ is a constant depending only on the charge $q_0$ and $\ga$. The above estimate implies that due to the better decay of the initial data, we can improve the $r$-weighted energy
estimate with the largest $p$ value for smaller $r_2$. For those points when $r_2$ is large, we rely on the negative $r$ weights in \eqref{Sobpsi}.

\bigskip

 Fix $u$ and consider the outgoing null hypersurface $S_{r_1, r_2}$ for
$r_2>r_1^*=r_1^{1+\frac{\ep_0}{C_1}}$. For $v\geq u+r_1^*$, let $v_1=u+r_1^*$ in the Sobolev embedding \eqref{Sobpsi}. We then can estimate that
\begin{align*}
 \int_{S_{r_1, r_2}}r^{\ga-2}|\psi|^2dvd\om&\les \int_{S_{r_1, r_1*}}r^{\ga-2} |\psi|^2dvd\om+(r_1^*)^{\ga-1}\int_{\om}|\psi|^2(u, -u, \om)dvd\om\\
&+(r_1^*)^{\ga-1}r_1^{-\ga}\int_{S_{r_1, r_1^*}}r^{1+\ga}|D_L\psi|^2dvd\om+(r_1^*)^{-1}\int_{S_{r_1, r_2}}r^{1+\ga}|D_L\psi|^2dvd\om\\
&\les r_1^{\ga-1}\int_{\om}|\psi|^2(u, -u, \om)dvd\om+r_1^{-1}\int_{S_{r_1, r_1^*}}r^{1+\ga}|D_L\psi|^2dvd\om\\
&\quad+r_1^{-1-\frac{\ep_0}{C_1}}\int_{S_{r_1, r_2}}r^{1+\ga}|D_L\psi|^2dvd\om\\
&\les r_1^{\ga-1}\int_{\om}|\psi|^2(u, -u, \om)dvd\om+E_0r_1^{-1-\ep_0}\\
&\quad+r_1^{-1-\frac{\ep_0}{C_1}}\int_{S_{r_1, r_2}}r^{1+\ga}|D_L\psi|^2dvd\om.
\end{align*}
Here note that $r_1^*\geq r_1$ and $\ga\leq 1$. In the last line of the previous estimate, the first term can be bounded by the assumption on the initial data. The second term is integrable in terms of $r_1$. The
improved decay rate in the last term allows us to use Gronwall's inequality. We comment here that the above estimate holds for all $r_2\geq r_1$ as when $r_2\leq r_1^*$ the integral on the
left hand side can be controlled by the first two terms on the right hand side. Now back to the $r$-weighted energy identity \eqref{pWescaout}, for $r_2\geq r_1*$ and $p=1+\ga$, the above calculations imply that
\begin{align*}
 \int_{S_{r_1, r_2}}r^{1+\ga}|D_L\psi|^2dvd\om\leq C_2 E_0 r_1^{-\ep_0}+C_2 \int_{r_1}^{r_2}s^{-1-\frac{\ep_0}{C_1}}\int_{S_{s, r_2}}r^{1+\ga}|D_L\psi|^2dvd\om ds
\end{align*}
for a constant $C_2$ depending only on the charge $q_0$ and $\ga$. The above estimate holds for all $R\leq r_1\leq r_2$. Gronwall's inequality then implies that
\begin{equation}
 \label{pWe1ga}
\int_{S_{r_1, r_2}}r^{1+\ga}|D_L\psi|^2dvd\om+\int_{r_1}^{r_2}s^{-1-\frac{\ep_0}{C_1}}\int_{S_{s, r_2}}r^{1+\ga}|D_L\psi|^2dvd\om ds\leq C_3 E_0 r_1^{-\ep_0}
\end{equation}
for some constant $C_3$ depending only on $q_0$, $\ga$. This estimate is sufficient to bound the error term arising from the charge. From \eqref{SovpsiI}, we have
\begin{align*}
 |q_0\iint_{\mathcal{D}_{r_1, r_2}}r^{p-2}J_L dxdt|&\leq \frac{p}{2}\iint_{\mathcal{D}_{r_1, r_2}}r^{p-1}|D_L\psi|^2dudvd\om+\frac{8q_0^2}{p}\iint_{\mathcal{D}_{r_1, r_2}}r^{p-3}|\psi|^2dudvd\om\\
&\leq \frac{p}{2}\iint_{\mathcal{D}_{r_1, r_2}}r^{p-1}|D_L\psi|^2dudvd\om+C_4 E_{r_1, r_2}^{p}+C_4 E_0 r_1^{-\ga+p-1-\ep_0}\\
& \leq\frac{p}{2}\iint_{\mathcal{D}_{r_1, r_2}}r^{p-1}|D_L\psi|^2dudvd\om+2C_4 E_0 r_1^{-\ga_0+p-1},\quad p\leq 1+\ga
\end{align*}
for some constant $C_4$ depending only on $q_0$, $\ga$. Then from the $r$-weighted energy identity \eqref{pWescaout}, we derive the final $r$-weighted energy estimate we need:
\begin{equation}
\label{pWeMKGout1}
\begin{split}
&  \iint_{\mathcal{D}_{r_1, r_2}}r^{p-1}(|D_L\psi|^2+|\D\psi|^2)+r^{p+1}(|\a|^2+|\tilde{\rho}|^2+\si^2)dvd\om du\\
&+\int_{S_{r_1,r_2}}r^p(|D_L\psi|^2+r^2|\a|^2)dvd\om+\iint_{\mathcal{D}_{r_1, r_2}}r^{p-3}|\psi|^2dudvd\om
\les E_0 r_1^{p-1-\ga_0}
\end{split}
\end{equation}
for all $R\leq r_1\leq r_2$, $0\leq p\leq 1+\ga<1+\ga_0$. Here based on our notation, the implicit constant depends only on $q_0$, $\ga$, $\ga_0$. The estimate for the integral of $|\psi|^2$ is
derived from the previous estimate.

\bigskip

Once we have the $r$-weighted energy estimate as well as the control of the integral of $|\phi|^2$, we can improve the integrated local energy estimates and the energy estimates. Take $\tilde{F}=F-\bar F$
in the integrated local energy estimate \eqref{ILEMKGout} and the energy estimate \eqref{esout}. We have
\[
 |\tilde{F}-F|=|\bar F|\leq |q_0|r^{-2}.
\]
Take $p=\ep$ in the $r$-weighted energy estimate \eqref{pWeMKGout1}. We get
\begin{equation*}
\begin{split}
\iint_{\mathcal{D}_{r_1, r_2}}(1+r)^{1+\ep}|F_{r\nu}-\tilde{F}_{r\nu}|^2|\phi|^2dxdt&\les \iint_{\mathcal{D}_{r_1, r_2}}(1+r)^{\ep-3}|\psi|^2dudvd\om\\
&\les E_0 r_1^{\ep-1-\ga_0}.
\end{split}
\end{equation*}
Here recall that $\psi=r\phi$. Then from the integrated local energy estimate \eqref{ILEMKGout} and the energy estimate \eqref{esout}, we can show that
\begin{align*}
 &r_1^\ep\iint_{\mathcal{D}_{r_1, r_2}}\frac{|\bar D\phi|^2+|\tilde{F}|^2}{(1+r)^{1+\ep}}+\frac{|\D\phi|^2+\tilde\rho^2+\si^2}{1+r}dxdt\\
&\les E[\phi, \tilde{F}](B_{r_1, \infty})+E_0r_1^{-\ep-1-\ga_0}+\iint_{\mathcal{D}_{r_1, r_2}}|F_{0\nu}-\tilde{F}_{0\nu}||J^\nu|dxdt\\
&\les E_0r_1^{-\ep-1-\ga_0}+\iint_{\mathcal{D}_{r_1, r_2}}\ep_1 r_1^\ep\frac{|\bar D\phi|^2}{(1+r)^{1+\ep}}+\ep_1^{-1}r_1^{-\ep}(1+r)^{\ep-3}|\phi|^2dxdt\\
&\les E_0r_1^{-1-\ga_0}(1+\ep_1^{-1})+\ep_1 r_1^\ep\iint_{\mathcal{D}_{r_1, r_2}}\frac{|\bar D\phi|^2}{(1+r)^{1+\ep}}dxdt
\end{align*}
for all $\ep_1>0$. Here notice that
\[
 E[\phi, \tilde{F}](B_{r_1, \infty})\leq E_0 r_1^{-1-\ga_0}.
\]
We also remark here that the regularity of the initial data is propagated in the exterior region as $E[\phi, \tilde{F}](B_{r_1, \infty})$ is the dominant term in the above estimate.
Take $\ep_1$ to be sufficiently small depending only on $q_0$, $\ep$, $\ga$, $\ga_0$ so that the last term in the last line can be absorbed. We derive the integrated local energy
estimate in the exterior region
\begin{align*}
 &r_1^\ep\iint_{\mathcal{D}_{r_1, r_2}}\frac{|\bar D\phi|^2+|\tilde{F}|^2}{(1+r)^{1+\ep}}+\frac{|\D\phi|^2+\tilde\rho^2+\si^2}{1+r}dxdt\les E_0r_1^{-1-\ga_0}.
\end{align*}
The previous estimate also gives control of the integral of $|F_{0\mu}-\tilde{F}_{0\mu}||J^\mu|$. Hence from the energy estimate \eqref{esout}, we obtain the energy control of the solution in the exterior region
\begin{equation}
\label{esMKGout1}
 E[\phi, \tilde{F}](S_{r_1, r_2})+E[\phi, \tilde{F}](\bar C_{r_1, r_2})\les E_0r_1^{-1-\ga_0}
\end{equation}
for all $R\leq r_1<r_2$. Let $r_2\rightarrow \infty$. We obtain the decay estimates for the chargeless part of the solution in the exterior region.

\section{Energy estimates in the interior region}
In the previous section, we obtained the energy estimates in the exterior region. In particular, we have estimates for the energy flux as well as the $r$-weighted energy
flux through $S_{R, \infty}$ or $S_0$ which is the outgoing null hypersurface starting from the sphere with radius $R$ on the initial hypersurface. These boundary information allows us to obtain
energy estimates for the solutions of the MKG equations in the interior region with foliation $\Si_{\tau}$.

\bigskip

From the $r$-weighted energy estimate \eqref{pWeMKGout1} and the energy estimate \eqref{esMKGout1} in the exterior region, we in particular have
\begin{equation*}
\begin{split}
&\int_{S_{R, r_2}}r^{1+\ga}(|D_L\psi|^2+r^2|\a|^2)dvd\om\les E_0 R^{\ga-\ga_0}\les E_0, \\
& E[\phi, F](S_{R, r_2})\les E[\phi, \tilde{F}](S_{R, r_2})+\int_{S_{R, r_2}}q_0^2 r^{-4}r^2dvd\om\les E_0
\end{split}
\end{equation*}
for all $r_2> R$. Since the implicit constant is independent of $r_2$, take $r_2\rightarrow \infty$. We obtain
\begin{equation}
 \label{bdcin}
\int_{S_{0}}r^{1+\ga}(|D_L\psi|^2+r^2|\a|^2)dvd\om+ E[\phi, F](\Si_{0})\les E_0.
\end{equation}
In the interior region we take $\tilde{F}=F$. From the integrated local energy estimate \eqref{ILEMKGin} and the energy estimate \eqref{esin} in the interior region, we obtain
\begin{equation}
 \label{ILEMKGin1}
\begin{split}
 & E[\phi, F](\Si_{\tau_2})+E[\phi, F](\bar C_{\tau_1, \tau_2})
+\int_{\tau_1}^{\tau_2}\int_{\Si_{\tau}}\frac{|\bar D\phi|^2+|F|^2}{(1+r)^{1+\ep}}+\frac{|\D\phi|^2+\rho^2+\si^2}{1+r}dxdt\\
&\les  E[\phi, F](\Si_{\tau_1})
\end{split}
\end{equation}
for all $0\leq \tau_1<\tau_2$. The improved integrated local energy estimate for $\D\phi$, $\rho$, $\si$ will be used to control the
boundary terms in the $r$-weighted energy estimates. Since $\tilde{F}=F$, the $r$-weighted energy identity \eqref{pWescain} implies that
\begin{align}
\notag
 &\int_{S_{\tau_2}}r^p|D_L\psi|^2+r^{p+2}|\a|^2dvd\om+\int_{\bar C_{\tau_1, \tau_2}}r^p(|\D\psi|^2+r^2\rho^2+r^2\si^2)dud\om\\
\label{pWeMKGin}
&+\int_{\tau_1}^{\tau_2}\int_{S_\tau}r^{p-1}\left(p(|D_L\psi|^2+r^2|\a|^2)+(2-p)(|\D\psi|^2+r^2\rho^2+r^2\si^2)\right)dvd\om d\tau\\
\notag
=&\int_{S_{\tau_1}}r^p|D_L\psi|^2+r^{p+2}|\a|^2dvd\om-\f12 R^p\int_{\tau_1}^{\tau_2}\int_{\om}|D_L\psi|^2-|\D\psi|^2+R^2(|\a|^2-\rho^2-\si^2) d\om dt.
\end{align}
For $p\in[0, 2]$, every term in the above identity has a positive sign except the last term which is the integral on the boundary of the cylinder with radius $R$. However, we observe that
that term is a constant multiple of $R^p$ which means that we can simply take $p=0$ in the above $r$-weighted energy identity in order to estimate the boundary term. From the above energy estimate \eqref{ILEMKGin1}
and the bound \eqref{bdcin}, we conclude that the energy flux through the null infinity $\bar C_{\tau_1, \tau_2}$ is finite. Since the scalar field $\phi$ vanishes at the spatial infinity initially,
we have that the scalar field must also vanish at null infinity. Then using Sobolev embedding and the relation \eqref{eqfact}, on $S_{\tau}$, $\tau\geq0$, we have
\[
 r\int_{\om}|\phi|^2d\om \leq \int_{S_\tau}|D_L\phi|^2r^2dvd\om\leq E[\phi, F](\Si_{\tau}).
\]
Here similar to Lemma \ref{Hardy} of Hardy's inequality, we can choose a particular gauge so that $\phi$ is real. Therefore the above estimate holds for all connection $D$. Note that
\[
 |D_L\psi|^2=r^2|D_L\phi|^2+L(r|\phi|^2),\quad \psi=r\phi.
\]
The previous estimate then implies that
\begin{align*}
 \int_{S_\tau}|D_L\psi|^2dvd\om \leq \int_{S_{\tau}}|D_L\phi|^2r^2dvd\om+\int_{\om}r|\phi|^2d\om\leq 2E[\phi, F](\Si_{\tau}).
\end{align*}
Note that $\D\psi=r\D\phi$. Let $p=0$ in the above $r$-weighted energy identity \eqref{pWeMKGin}. From the energy estimate \eqref{ILEMKGin1}, we can estimate the boundary terms as follows:
\begin{align*}
& \left|\f12 R^p\int_{\tau_1}^{\tau_2}\int_{\om}|D_L\psi|^2-|\D\psi|^2+R^2(|\a|^2-\rho^2-\si^2) d\om dt\right|\\
&\leq 2E[\phi, F](\Si_{\tau_1})+2E[\phi, F](\Si_{\tau_2})+E[\phi, F](\bar C_{\tau_1, \tau_2})+2\int_{\tau_1}^{\tau_2}\int_{S_\tau}\frac{|\D\phi|^2+\rho^2+\si^2}{r}dxdt\\
&\les E[\phi, F](\Si_{\tau_1}).
\end{align*}
Therefore for $0<p<2$, the $r$-weighted energy identity \eqref{pWeMKGin} implies that
\begin{align}
 \notag
 &\int_{S_{\tau_2}}r^p(|D_L\psi|^2+r^{2}|\a|^2)dvd\om+\int_{\tau_1}^{\tau_2}\int_{S_\tau}r^{p-1}(|D_L\psi|^2+|\D\psi|^2+r^2(|\a|^2+\rho^2+\si^2))dvd\om d\tau\\
\label{pWeMKGineq}
\les&\int_{S_{\tau_1}}r^p|D_L\psi|^2+r^{p+2}|\a|^2dvd\om+ E[F, \phi](\Si_{\tau_1}).
\end{align}
Here the implicit constant depends also on $R$ and $p$. Take $p=1+\ga$ in the above $r$-weighted energy inequality. The boundary condition \eqref{bdcin} shows that
\begin{align}
\label{pWeMKGinga}
&\int_{S_{\tau_2}}r^{1+\ga}|D_L\psi|^2+r^{\ga+3}|\a|^2dvd\om\\
\notag
&+\int_{\tau_1}^{\tau_2}\int_{S_\tau}r^{\ga}(|D_L\psi|^2+|\D\psi|^2)+r^{2+\ga}(|\a|^2+\rho^2+\si^2)dvd\om d\tau\les E_0.
\end{align}
We want to retrieve the integral of the energy flux through the whole hypersurface $\Si_{\tau}$ from the $r$-weighted energy estimate with $p=1$. We thus can make use of the integrated local energy
estimate \eqref{ILEMKGin1} restricted to the region $r\leq R$. First we note that when $r\leq R$, using Sobolev embedding, we have
\[
\int_{\om}|\phi|^2d\om\les \int_{r\leq R}|D\phi|^2+|\phi|^2dx.
\]
Then from the $r$-weighted energy estimate \eqref{pWeMKGineq} with $p=1$ we can show that
\begin{align}
\notag
\int_{\tau_1}^{\tau_2}E[\phi, F](\Si_{\tau})d\tau&=\int_{\tau_1}^{\tau_2}\int_{r\leq R}|D\phi|^2+|F|^2dxdt+\int_{\tau_1}^{\tau_2}\int_{\om}|\phi|^2(\tau, R, \om)d\om d\tau\\
 \label{pWeMKGin1}
&\quad +\int_{\tau_1}^{\tau_2}\int_{S_\tau}|D_L\psi|^2+|\D\psi|^2+r^2(|\rho|^2+|\a|^2+|\si|^2) dvd\om d\tau\\
\notag
&\les \int_{S_{\tau_1}}r|D_L\psi|^2+r^{3}|\a|^2dvd\om+ E[\phi, F](\Si_{\tau_1})\\
\notag
&\quad+\int_{\tau_1}^{\tau_2}\int_{r\leq R}|\bar D\phi|^2+|F|^2dxdt\\
\notag
&\les  \int_{S_{\tau_1}}r|D_L\psi|^2+r^{3}|\a|^2dvd\om+ E[\phi, F](\Si_{\tau_1}).
\end{align}
The $r$-weighted energy
inequality \eqref{pWeMKGinga} implies that there is a dyadic sequence $\{\tau_{n}\}$ such that
\[
 \int_{S_{\tau_{n}}}r^{\ga}|D_L\psi|^2+r^{\ga+2}|\a|^2dvd\om\les E_0 (1+\tau_n)^{-1},\quad \forall n.
\]
Since \eqref{pWeMKGinga} also implies that
\[
 \int_{S_{\tau}}r^{1+\ga}|D_L\psi|^2+r^{\ga+3}|\a|^2dvd\om\les E_0,\quad \forall \tau\geq 0,
\]
interpolation implies that
\[
 \int_{S_{\tau_{n}}}r|D_L\psi|^2+r^{3}|\a|^2dvd\om\les E_0 (1+\tau)^{-\ga}, \quad \forall n.
\]
The energy estimate \eqref{ILEMKGin1} implies that the energy flux $E[\phi, F](\Si_\tau)$ is nonincreasing with respect to $\tau$. Then from the $r$-weighted energy estimate \eqref{pWeMKGin1} for $p=1$, we
can show that
\begin{align*}
(\tau-\tau_n)E[\phi, F](\Si_{\tau})\les E_0 (1+\tau_n)^{-\ga}+E[\phi, F](\Si_{\tau_n}),\quad \forall \tau\geq \tau_n.
\end{align*}
In particular, for $n=0$, we obtain
\[
 E[\phi, F](\Si_{\tau})\les E_0(1+\tau)^{-1},\quad \forall \tau\geq 0.
\]
Since the sequence $\{\tau_n\}$ is dyadic, for $\tau\in [\tau_{n+1}, \tau_{n+2}]$, we have
\[
 E[\phi, F](\Si_{\tau})\les (\tau-\tau_n)^{-1}(E_0(1+\tau_n)^{-\ga}+E_0(1+\tau_n)^{-1})\les E_0(1+\tau)^{-1-\ga}.
\]
From the integrated local energy estimate \eqref{ILEMKGin1}, the above decay of the energy flux then implies the decay of the integrated local energy in the interior region. This finishes the proof of the
main Theorem.

\bibliography{shiwu}{}

\begin{thebibliography}{10}

\bibitem{bluemaxkerr}
L.~Andersson and P.~Blue.
\newblock {Uniform energy bound and asymptotics for the Maxwell field on a
  slowly rotating Kerr black hole exterior}.
\newblock 2013.
\newblock ar{X}iv: 1310.2664.

\bibitem{yannisER}
Y.~Angelopoulos.
\newblock {Nonlinear Wave Equations With Null Condition On Extremal
  Reissner-Nordström Spacetimes I: Spherical Symmetry }.
\newblock 2014.
\newblock ar{X}iv:1408.4478.

\bibitem{MKGLydia}
L.~Bieri, S.~Miao, and S.~Shahshahani.
\newblock {Asymptotic properties of solutions of the Maxwell Klein Gordon
  equation with small data}.
\newblock 2014.
\newblock ar{X}iv: 1408.2550.

\bibitem{fieldschrist}
Y.~Choquet-Bruhat and D.~Christodoulou.
\newblock Existence of global solutions of the {Y}ang-{M}ills, {H}iggs and
  spinor field equations in {$3+1$} dimensions.
\newblock {\em Ann. Sci. \'Ecole Norm. Sup. (4)}, 14(4):481--506 (1982), 1981.

\bibitem{asymLkl}
D.~Christodoulou and S.~Klainerman.
\newblock Asymptotic properties of linear field equations in {M}inkowski space.
\newblock {\em Comm. Pure Appl. Math.}, 43(2):137--199, 1990.

\bibitem{dr3}
M.~Dafermos and I.~Rodnianski.
\newblock The redshift effect and radiation decay on black hole spacetimes.
\newblock {\em Comm. Pure Appl. Math.}, 62(7):859--919, 2009.

\bibitem{newapp}
M.~Dafermos and I.~Rodnianski.
\newblock A new physical-space approach to decay for the wave equation with
  applications to black hole spacetimes.
\newblock In {\em X{VI}th {I}nternational {C}ongress on {M}athematical
  {P}hysics}, pages 421--432. World Sci. Publ., Hackensack, NJ, 2010.

\bibitem{newapp3}
M.~Dafermos, I.~Rodnianski, and Y.~Shlapentokh-Rothman.
\newblock {Decay for solutions of the wave equation on Kerr exterior spacetimes
  {III}: {T}he case {$|a|<M$}}.
\newblock 2014.
\newblock ar{X}iv:1402.7034.

\bibitem{Moncrief1}
D.~Eardley and V.~Moncrief.
\newblock The global existence of {Y}ang-{M}ills-{H}iggs fields in
  {$4$}-dimensional {M}inkowski space. {I}. {L}ocal existence and smoothness
  properties.
\newblock {\em Comm. Math. Phys.}, 83(2):171--191, 1982.

\bibitem{Moncrief2}
D.~Eardley and V.~Moncrief.
\newblock The global existence of {Y}ang-{M}ills-{H}iggs fields in
  {$4$}-dimensional {M}inkowski space. {II}. {C}ompletion of proof.
\newblock {\em Comm. Math. Phys.}, 83(2):193--212, 1982.

\bibitem{gustavUSchwarz}
G.~Holzegel.
\newblock {Ultimately Schwarzschildean Spacetimes and the Black Hole Stability
  Problem}.
\newblock 2010.
\newblock ar{X}iv:1010.3216.

\bibitem{MKGtao}
M.~Keel, T.~Roy, and T.~Tao.
\newblock Global well-posedness of the {M}axwell-{K}lein-{G}ordon equation
  below the energy norm.
\newblock {\em Discrete Contin. Dyn. Syst.}, 30(3):573--621, 2011.

\bibitem{MKGkl}
S.~Klainerman and M.~Machedon.
\newblock On the {M}axwell-{K}lein-{G}ordon equation with finite energy.
\newblock {\em Duke Math. J.}, 74(1):19--44, 1994.

\bibitem{YMkl}
S.~Klainerman and M.~Machedon.
\newblock Finite energy solutions of the {Y}ang-{M}ills equations in
  {$\mathbb{R}^{3+1}$}.
\newblock {\em Ann. of Math. (2)}, 142(1):39--119, 1995.

\bibitem{KriegerMKG4}
J.~Krieger and J.~Luhrmann.
\newblock {Concentration Compactness for the Critical Maxwell-Klein-Gordon
  Equation }.
\newblock 2015.
\newblock ar{X}iv:math.{A}{P}/1503.09101.

\bibitem{MKGtataru}
J.~Krieger, J.~Sterbenz, and D.~Tataru.
\newblock Global well-posedness for the {M}axwell-{K}lein-{G}ordon equation in
  {$4+1$} dimensions: small energy.
\newblock {\em Duke Math. J.}, 164(6):973--1040, 2015.

\bibitem{LindbladMKG}
H.~Lindblad and J.~Sterbenz.
\newblock Global stability for charged-scalar fields on {M}inkowski space.
\newblock {\em IMRP Int. Math. Res. Pap.}, pages Art. ID 52976, 109, 2006.

\bibitem{MKGmachedon}
M.~Machedon and J.~Sterbenz.
\newblock Almost optimal local well-posedness for the {$(3+1)$}-dimensional
  {M}axwell-{K}lein-{G}ordon equations.
\newblock {\em J. Amer. Math. Soc.}, 17(2):297--359 (electronic), 2004.

\bibitem{mora1}
C.~S. Morawetz.
\newblock The limiting amplitude principle.
\newblock {\em Comm. Pure Appl. Math.}, 15:349--361, 1962.

\bibitem{mora2}
C.~S. Morawetz.
\newblock Time decay for the nonlinear klein-gordon equations.
\newblock {\em Proc. Roy. Soc. Ser. A}, 306:291--296, 1968.

\bibitem{sungjinYM}
S.~Oh.
\newblock Finite energy global well-posedness of the {Y}ang-{M}ills equations
  on {$\Bbb{R}^{1+3}$}: an approach using the {Y}ang-{M}ills heat flow.
\newblock {\em Duke Math. J.}, 164(9):1669--1732, 2015.

\bibitem{OhMKG4}
S.~Oh and D.~Tataru.
\newblock {Global well-posedness and scattering of the (4+1)-dimensional
  Maxwell-Klein-Gordon equation}.
\newblock 2015.
\newblock ar{X}iv:math.{A}{P}/1503.01562.

\bibitem{MKGigor}
I.~Rodnianski and T.~Tao.
\newblock Global regularity for the {M}axwell-{K}lein-{G}ordon equation with
  small critical {S}obolev norm in high dimensions.
\newblock {\em Comm. Math. Phys.}, 251(2):377--426, 2004.

\bibitem{volkerSch}
V.~Schlue.
\newblock Decay of linear waves on higher-dimensional {S}chwarzschild black
  holes.
\newblock {\em Anal. PDE}, 6(3):515--600, 2013.

\bibitem{MKGtesfahun}
S.~Selberg and A.~Tesfahun.
\newblock Finite-energy global well-posedness of the {M}axwell-{K}lein-{G}ordon
  system in {L}orenz gauge.
\newblock {\em Comm. Partial Differential Equations}, 35(6):1029--1057, 2010.

\bibitem{Shu}
W.~Shu.
\newblock Asymptotic properties of the solutions of linear and nonlinear spin
  field equations in {M}inkowski space.
\newblock {\em Comm. Math. Phys.}, 140(3):449--480, 1991.

\bibitem{shu2}
W.~Shu.
\newblock Global existence of {M}axwell-{H}iggs fields.
\newblock In {\em Nonlinear hyperbolic equations and field theory ({L}ake
  {C}omo, 1990)}, volume 253 of {\em Pitman Res. Notes Math. Ser.}, pages
  214--227. Longman Sci. Tech., Harlow, 1992.

\bibitem{tatarumaxile}
J.~Sterbenz and D.~Tataru.
\newblock {Local energy decay for Maxwell fields part I: Spherically symmetric
  black-hole backgrounds}.
\newblock 2013.
\newblock ar{X}iv: 1305.5621.

\bibitem{lodecTataru}
D.~Tataru.
\newblock Local decay of waves on asymptotically flat stationary space-times.
\newblock {\em Amer. J. Math.}, 135(2):361--401, 2013.

\bibitem{qiansharpLocal}
Q.~Wang.
\newblock {A geometric approach for sharp Local well-posedness of quasilinear
  wave equations}.
\newblock 2014.
\newblock ar{X}iv:1408.3780.

\bibitem{yang1}
S.~Yang.
\newblock Global solutions of nonlinear wave equations in time dependent
  inhomogeneous media.
\newblock {\em Arch. Ration. Mech. Anal.}, 209(2):683--728, 2013.

\bibitem{yang3}
S.~Yang.
\newblock {On the quasilinear wave equations in time dependent inhomogeneous
  media}.
\newblock 2013.
\newblock ar{X}iv:1312.7264.

\bibitem{yang5}
S.~Yang.
\newblock Global solutions of nonlinear wave equations with large data.
\newblock {\em Selecta Math. (N.S.)}, 21(4):1405--1427, 2015.

\bibitem{yang2}
S.~Yang.
\newblock Global stability of solutions to nonlinear wave equations.
\newblock {\em Selecta Math. (N.S.)}, 21(3):833--881, 2015.

\end{thebibliography}
\bibliographystyle{plain}

\bigskip

DPMMS, Centre for Mathematical Sciences, University of Cambridge,
Wilberforce Road, Cambridge, UK CB3 0WA

\textsl{Email address}: S.Yang@dpmms.cam.ac.uk
 \end{document}